\newif\ifmor
\morfalse

\ifmor
\documentclass[moor]{informs1}
\else
\documentclass[a4paper]{article}
\fi

\usepackage{amsmath}
\usepackage{amsfonts}
\usepackage{amssymb}         
\usepackage{amsopn}
\usepackage{theorem}          
\usepackage{latexsym}
\usepackage{graphicx}        
\usepackage{mathrsfs}		
\usepackage{psfrag}
\usepackage{algorithmicx}
\usepackage{algorithm}
\usepackage[noend]{algpseudocode}
\usepackage{color}
\usepackage{verbatim}
\usepackage{enumerate}
\usepackage{url}
\usepackage{pgf,tikz}
\usepackage{mathrsfs}
\usetikzlibrary{arrows}
\definecolor{plgreen}{rgb}{0,0.5,0}
\newcommand{\PL}[1]{{\color{black}#1}}
\newcommand{\LL}[1]{{\color{black}#1}}
\newcommand{\Ky}[1]{{\color{black}#1}}
\definecolor{darkgrey}{gray}{0.4}

\ifmor
\usepackage{natbib}
 \NatBibNumeric
 \bibpunct[, ]{[}{]}{,}{n}{}{,}%
 
\usepackage[colorlinks=true,breaklinks=true,bookmarks=true,urlcolor=blue,
     citecolor=blue,linkcolor=blue,bookmarksopen=false,draft=false]{hyperref}

\TheoremsNumberedThrough     
\EquationsNumberedThrough    

\else
\newtheorem{theorem}{Theorem}[section]
\newtheorem{lemma}[theorem]{Lemma}
\newtheorem{proposition}[theorem]{Proposition}
\newtheorem{corollary}[theorem]{Corollary}

\fi

\ifmor
\else

\newenvironment{proof}
 {{\sl Proof.}\hspace*{1 ex}}%
 {{\nopagebreak\hspace*{\fill}$\Box$\par\vspace{12pt}}}
\newcommand{\qed}{\hfill \ensuremath{\Box}}
\DeclareMathOperator{\argmin}{\mbox{arg\,min}}
\fi

\newcommand{\transpose}[1]{{#1}^{\top}}

\newcommand{\cone}[1]{\mbox{\sf cone}(#1)}
\newcommand{\conv}[1]{\mbox{\sf conv}(#1)}
\renewcommand{\ker}[1]{\mbox{\sf ker}(#1)}
\renewcommand{\dim}[1]{\mbox{\sf dim}(#1)}
\newcommand{\MyAbstract}{Random projections are random linear maps, sampled from appropriate distributions, that approximately preserve certain geometrical invariants so that the approximation improves as the dimension of the space grows. The well-known Johnson-Lindenstrauss lemma states that there are \LL{random matrices with surprisingly few rows} that approximately preserve pairwise Euclidean distances among a set of points. This is commonly used to speed up algorithms based on Euclidean distances. We prove that these matrices also preserve other quantities, such as the distance to a cone. We exploit this result to devise a probabilistic algorithm to solve linear programs approximately. \LL{We show that this algorithm can approximately solve very large randomly generated LP instances. We also showcase its application to an error correction coding problem.}}

\begin{document}

\ifmor
\else

\thispagestyle{empty}
\begin{center} 
{\LARGE Random projections for linear programming}
\par \bigskip
{\sc Vu Khac Ky, Pierre-Louis Poirion, Leo Liberti} 
\par \bigskip
{\it LIX, \'Ecole Polytechnique, F-91128 Palaiseau,
France} \\ Email:\url{{vu,poirion,liberti}@lix.polytechnique.fr}
\par \medskip \today
\end{center}
\par \bigskip
\date{\today}

\fi

\ifmor
\else
\begin{abstract}
\MyAbstract
\end{abstract}
\fi

\section{Introduction}
\label{s:intro}
A deep and surprising result, called the {\it Johnson-Lindenstrauss Lemma} (JLL) \cite{jllemma}, states that a set of high dimensional points can be  projected to a much lower dimensional space while keeping Euclidean distances approximately the same. The JLL was previously exploited in purely Euclidean distance based algorithms, such as $k$-means \cite{boutsidis2010} and $k$ nearest neighbours \cite{indyk}. The JLL has rarely been employed in mathematical optimization. The few occurrences are related to reasonably natural cases such as linear regression \cite{pilanci2014}, where the error minimization is encoded by means of a Euclidean norm. One reason for this is that the very proof of the JLL exploits rotational invariance, naturally exhibited by sets of distances, but which feasible sets commonly occurring in Linear Programming (LP), such as orthants, obviously do not. In this paper we lay the theoretical foundations of solving LPs approximately using random projections, and showcase their usefulness in practice. More precisely, we address LPs in standard form
\begin{equation}
P\equiv\min\{\transpose{c}x\;|\;Ax=b\land x\in\mathbb{R}^n_+\}, \label{orig}
\end{equation}
where $A$ is an $m\times n$ matrix. For each $i\le m$ we let $A^i$ be the $i$-th row of $A$, and for each $j\le n$ we let $A_j$ be the $j$-th column of $A$. If $I$ is a set of row indices, we indicate the submatrix of $A$ consisting of those rows by $A^I$; if $J$ is a set of column indices, we indicate the submatrix of $A$ consisting of those columns by $A_J$. We let $\cone{A}$ be the cone spanned by the column vectors $A_j$ (for $j\le n$), and $\conv{A}$ be the convex hull of the column vectors $A_j$ (for $j\le n$). We denote by $v(P)$ the optimal objective function value of \Ky{the problem} $P$, and by $\mathcal{F}(P)$ its feasible region. Note that determining whether $\mathcal{F}(P)\not=\varnothing$ is exactly the same problem as determining whether $b\in\cone{A}$. Throughout this paper, all norms will be Euclidean unless specified otherwise.

In this paper we often assume that $b$, $c$, and all the column vectors of
$A$ have unit Euclidean norm. This assumption does not lose
generality: if $\tilde{x}$ is the optimal solution of the
reformulation of Eq.~\eqref{orig} where all these vectors have unit
norm, we can compute the optimal solution $x^\ast$ of Eq.~\eqref{orig}
as follows:\Ky{
\[\forall j \in \{1,\ldots,n\}, \quad x^\ast_j=\frac{  \|b\|\tilde{x}_j}{\|A_j\|}.\]
}
A {\it random projector} is a $k\times m$ matrix $T$, sampled from
appropriate distributions (more details on this below), which
preserves certain geometrical properties of sets of points in
$\mathbb{R}^m$. We denote by
\begin{equation}
P_T\equiv\min\{\transpose{c}x\;|\;TAx=Tb\land
x\in\mathbb{R}^n_+\} \label{proj}
\end{equation}
the random projection version of \Ky{the problem} $P$. Our main result
(Thm.~\ref{thm:objfunapprox}) is that we can construct a random
projector $T$, with $k\ll n$, such that, for some given
$\varepsilon>0$, we have $|v(P)-v(P_T)|\le\varepsilon$ with
overwhelming probability (w.o.p.). Moreover, for fixed $\varepsilon$,
$k$ turns out to be $O(\ln n)$. Since the complexity of solving LPs
depends on both $m$ and $n$, a logarithmic reduction on $m$ (even as a
function of $n$) appears very appealing. By ``w.o.p.'' we mean that
the probability of the concerned event is $1-f(k)$ where $f(k)$ tends
to zero extremely fast as $k$ tends to infinity. Typically, $f$ is $O(e^{-k})$.

So far so good; unfortunately, there are some bad news too. First, we
prove that the optimum of \Ky{the projected problem} $P_T$ is
infeasible w.r.t.~$\mathcal{F}(P)$ (the original region) with
probability 1 (Prop.~\ref{prop:certificate1}), which appears to
severely limit the usefulness of Thm.~\ref{thm:objfunapprox} --- we
address this limitation in Sect.~\ref{s:retrieval}. Second, sampling
$T$ and performing matrix multiplications $T(A,b)$ is time consuming,
since $T$ is a dense matrix. Third, even though the original LP is
sparse, the projected LP is dense as a result of $T$ being dense,
which means that solving it has an added computational cost. Last, but
not least, we have no idea about how to estimate, much less compute,
the constant in the term $O(\ln n)$. We know that the term
$\frac{1}{\varepsilon^2}$, which is large if we want the approximation
to be tight, plays a role; but there are other universal constants
that also play a role. We also know that the probability of the event
$|v(P)-v(P_T)|\le\varepsilon$ approaches 1 as $1-O(e^{-k})$. All this
suggests that any practical usefulness of this methodology will come
from very large instances and/or very dense instances.


\LL{
  \subsection{Differences with existing literature}
  \label{s:litrev}
Randomized dimension reduction techniques are widely used in the analysis of large data sets, but much less so in Mathematical Programming (MP). Specifically, in the field of LP we are aware of the three main results \cite{candestao2005,pilanci2014,pucci}. We set compressed sensing \cite{candestao2005} aside, as strictly speaking this is not a solution or reformulation method, but rather a theoretical analysis which explains why $\ell_1$-norm minimization of the error of an underdetermined linear system is an excellent proxy for reconstructing sparse solutions. Although we are only citing the paper \cite{candestao2005} for compressed sensing, this line of work gave rise to a very large number of papers by many different authors. We shall see in Sect.~\ref{s:coding} that compressed sensing can be ``further compressed'' using our methodology.

In \cite{pilanci2014}, it is shown how matrix sketching (which is strongly related to random projections) can help decrease the dimensionality of some convex quadratic minimization over an arbitrary convex set $\mathcal{C}$ from a given $\mathbb{R}^m$ to $\mathbb{R}^k$ for some $k\le m$. Prop.~\ref{prop:cvxhull} below emphasizes some of the differences with the present work; \cite[Eq.~(28) \S 3.4]{pilanci2014}, for example, encodes the problem of deciding whether zero is in the convex hull of the columns of a given matrix $B$. Unlike our development, the analysis provided in \cite{pilanci2014} requires the projected dimension $k$ to be bounded below by a function of several parameters before any probability estimation can be made. Another remarkable difference is that the framework described in \cite{pilanci2014} {\it requires} a convex purely quadratic objective function: to encode a linear objective $\transpose{c}x$ using a quadratic, the most direct way involves the introduction of a new scalar variable $y$, and then rewriting $\min \transpose{c}x$ as $\min y^2$ subject to $y\ge \transpose{c}x$ and $y\ge 0$. This reformulation, however, prevents the application of the method. Lastly, in \cite{pilanci2014} we find that the projected dimension $k$ is of the order of magnitude of the Gaussian width $W$ of $\mathcal{C}$. To require $k\ll m$, this implies working with convex sets $\mathcal{C}$ having small Gaussian width. By contrast, our technique optimizes over orthants, which have (large) Gaussian width $O(n)$.

The paper \cite{pucci} proposes a randomized dimensionality reduction based on PAC learning \cite{anthony}: from a small training set, it is possible to forecast some properties of large data sets while keeping the error low. This is exploited in LPs with very few variables and huge numbers of inequality constraints: it is found that this number can be greatly reduced while keeping the optimality error bounded. In order to have PAC learning assumptions work, the authors focus on application cases which have a specific structure, i.e.~there is an order on the constraints which makes their slope vary in a controlled way (an example is given by the piecewise linear approximation of a two-dimensional closed convex curve: one can take many tangents, but few of these suffice to give almost the same approximation). The prominent difference with the method proposed in this paper is that we make no such assumption.
}

\subsection{Contents}
The rest of the paper is organized as follows. Section \ref{s:jll} reports the basic concepts about the JLL. In Sect.~\ref{s:lpfeas} we show that random projections approximately preserve LP feasibility with high probability. The proof of our main theorem is offered in Sect.~\ref{s:opt}, where we argue that random projections also preserve LP optimality with high probability. In Sect.~\ref{s:retrieval} we address the limitation referred to above, and provide a method to work out the solution of the original LP given the solution of the projected LP. \LL{In Sect.~\ref{s:complexity} we make some remarks about computational complexity. Sect.~\ref{s:compres} reports some computational results, and Sect.~\ref{s:coding} showcases an application to error correcting codes.}


\section{The Johnson-Lindenstrauss lemma}
\label{s:jll}
The JLL is stated as follows:
\begin{theorem}[Johnson-Lindenstrauss Lemma \cite{jllemma}] \label{J-L lemma}
Given $\varepsilon \in (0,1)$ and an $m\times n$ matrix $A$, there exists a $k\times m$ matrix $T$ such that:
\begin{equation}
  \forall 1\le i<j\le n \quad (1-\varepsilon)\|A_i -A_j\| \le \|TA_i - TA_j\| \le
  (1+\varepsilon)\|A_i -A_j\|, \label{eq:jll}
\end{equation}
where $k$ is $O(\varepsilon^{-2}\ln n)$.
\end{theorem}
Thus, all sets of $n$ points can be projected to a subspace having dimension logarithmic in $n$ (and, surprisingly, independent of the original number $m$ of dimensions), such that no distance is distorted by more than $1+2\varepsilon$. The JLL can be established as a consequence of a general property (see Lemma \ref{randomprojectionlemma} below) of \emph{sub-gaussian} random mappings $T = \frac{1}{\sqrt{k}} U$ \cite{matousek-jll}. Some of the most popular choices for $U$ are:\\
\Ky{\textbf{Choices of random projection}:}
\begin{enumerate}
\item orthogonal projections on a random $k$-dimensional linear subspace of $\mathbb{R}^m$ \cite{jllemma};
\item random $k\times m$ matrices with each entry independently drawn from the standard normal distribution $\mathcal{N}(0,1)$ \cite{indyk};
\item random $k\times m$ matrices with each entry independently taking values $+1$ and $-1$, each with probability $\frac{1}{2}$ \cite{achlioptas};
\item random $k\times m$ matrices with entries independently taking values $+1$, $0$, $-1$, respectively with probability $\frac{1}{6}$, $\frac{2}{3}$, $\frac{1}{6}$ \cite{achlioptas} \LL{(we call this the {\it Achlioptas random projector}).}\label{enum:ach}
\end{enumerate}
Other, sparser projectors have been proposed in \cite{achlioptas,dasgupta2,kane,micali}. In this paper we just limit our attention to the normally distributed $T\sim\mathcal{N}(0,1/\sqrt{k})$ and its discrete approximation in Item \ref{enum:ach} above. Our reasons for ignoring this issue is that we believe that the rue bottleneck lies the unknown ``large constants'' referred to above. The matrix product operation (on which the choice of random projector would have the greatest impact) is one of the most common in scientific computing, and many ways are known to optimize and streamline it. \LL{In our computational experiments (Sect.~\ref{s:compres}-\ref{s:coding}) we use the Achlioptas projector and the most obvious matrix product implementation.}

Note that all the random projectors we consider have zero mean. This is necessary in order to ensure that our randomized algorithms will yield the result we want in expectation. This also explains why we consider LPs in standard rather than canonical form: we can not apply the random projection to the inequality system $Ax\le b$ to yield $TAx\le Tb$: \Ky{this is almost always false}, since the signs of the components of the matrix $T$ are distributed uniformly. 

The JLL can be derived from a more fundamental result \cite{matousekmetric}.
\begin{lemma}[Random projection lemma] \label{randomprojectionlemma}
For all $\varepsilon\in(0,1)$ and all vectors $y\in\mathbb{R}^m$, \Ky{let $T$ be a $k\times m$ random projector from one of the choices (1-4) above , then} 
\begin{equation} \label{eqn:01}
  \mbox{\sf Prob}(\,(1-\varepsilon)\|y\|\le\|Ty\|\le(1+\varepsilon)\|y\|\,)
  \ge 1-2e^{-\mathcal{C}\varepsilon^2 k}
\end{equation}
for some constant $\mathcal{C}>0$ (independent of $m,k,\varepsilon$).
\end{lemma}
\Ky {It can be proved easily that JLL is a consequence of Lemma \ref{randomprojectionlemma} by setting $y = A_i - A_j$ for all pairs of $(i,j)$ and then applying the union bound. Moreover, }Lemma \ref{randomprojectionlemma} shows that the probability of finding a good $T$ is very high for large enough values of $k$. Indeed, from Lemma \ref{randomprojectionlemma}, the probability that Eq.~\eqref{eq:jll} holds for all $i\not=j\le n$ is at least
\begin{equation}
  1 - 2 {n\choose 2}e^{-\mathcal{C}\varepsilon^2 k} =
    1 - n(n-1)e^{-\mathcal{C}\varepsilon^2 k}.\label{eq:probjll}
\end{equation}
Therefore, if we want this probability to be larger than, say $99.9\%$, we simply choose any $k$ such that $\frac{1}{1000n(n-1)}>e^{-\mathcal{C}\varepsilon^2 k}$.  This means $k$ can be chosen to be $k=\lceil\frac{\ln(1000)+2\ln(n)}{\mathcal{C}\varepsilon^2}\rceil$, which is $O(\,\varepsilon^{-2} (\ln(n)+3.5)\,)$. 

\Ky{Note that} the distributions from which $T$ is sampled are such that the the average of $\|Ty\|$ over $T$ is equal to $\|y\|$. Lemma \ref{randomprojectionlemma} is a {\it concentration of measure} result, and it states that the probability of a single sampling of $T$ yielding a value of $\|Ty\|$ very close to its mean approaches 1 as fast as a negative exponential of $k$ approaches zero.

We shall also need a squared version of the random projection lemma \cite{dasgupta}.
\begin{lemma}[Random projection lemma, squared version]
  \label{randomprojectionlemma2}
For all $\varepsilon\in(0,1)$ and all vectors $y\in\mathbb{R}^m$, \Ky{let $T$ be a $k\times m$ random projector from one of the choices (1-4) above,} then 
\begin{equation} \label{eqn:01a}
\mbox{\sf Prob}(\,(1-\varepsilon)\|y\|^2 \le \|Ty\|^2 \le (1+\varepsilon)\|y\|^2\,) \ge 1 - 2e^{-\mathcal{C}(\varepsilon^2-\varepsilon^3) k}
\end{equation}
for some constant $\mathcal{C} > 0$ (independent of $m,k,\varepsilon$).
\end{lemma}

Another relevant result about the JLL is the preservation of angles (or scalar product) \Ky{with high probability}. This result is not new, nor is it surprising in light of the JLL, but we report a proof here for completeness. \Ky{Indeed, given any $x, y \in \mathbb{R}^n$, and $T$ a $k\times m$ random projector from one of the choices (1-4) above, by applying Lemma   \ref{randomprojectionlemma2} on two vectors $x+y, \, x-y$ and using the union bound, we have
\begin{align*}
|\langle Tx, Ty \rangle  - \langle x, y \rangle| & = \tfrac{1}{4} \big|\|T(x+y)\|^2 - \|T(x-y)\|^2 -  \|x+y\|^2 + \|x-y\|^2 \big| \\
& \le  \tfrac{1}{4} \big|\|T(x+y)\|^2  -  \|x+y\|^2 \big| + \tfrac{1}{4}\big|\|T(x-y)\|^2 -  \|x-y\|^2\big| \\
& \le  \tfrac{\varepsilon}{4} (\|x +y\|^2 + \|x-y\|^2) = \tfrac{\varepsilon}{2} (\|x\|^2 + \|y\|^2),
\end{align*}
with probability at least $1 - 4 e^{-\mathcal{C}\varepsilon^2k}$. We can strengthen this further to obtain the following useful result.
\begin{proposition} \label{cor:angles}
	Let $T: \mathbb{R}^m \to \mathbb{R}^k$ be a $k\times m$ random projector from one of the choices (1-4) above  and let $0 < \varepsilon < 1$. Then there is a universal constant $\mathcal{C}$ such that, for any $x, y \in \mathbb{R}^n$:
	$$-\varepsilon \|x\| \, \|y\| \, \le \, \langle Tx, Ty \rangle - \langle x, y\rangle \, \le \, \varepsilon \|x\| \, \|y\|$$
	with probability at least $1 - 4e^{-\mathcal{C}\varepsilon^2 k}$.
\end{proposition}
\begin{proof} {\it Proof.}
	Apply the above result for $u= \frac{x}{\|x\|}$ and $v = \frac{y}{\|y\|}$. \qed
\end{proof}

From now on, when we say a ``random projector", we always mean a $k\times m$ random matrix from one of the choices (1-4) in Section 2.
}

\section{Preserving LP feasibility}
\label{s:lpfeas}
Consider the Linear Feasibility Problem (LFP)
\[\mathcal{F}=\mathcal{F}(P)\equiv\{x\in\mathbb{R}^n_+\;|\;Ax=b\}\]
and its randomly projected version
\[T\mathcal{F}=\mathcal{F}(P_T)\equiv\{x\in\mathbb{R}^n_+\;|\;TAx=Tb\}.\]
In this section we prove that $F\not=\varnothing$ if and only if $T\mathcal{F}\not=\varnothing$ w.o.p.

We remark that, for any $k\times m$ matrix $T$, any feasible solution for $\mathcal{F}$ is also a feasible solution for $T\mathcal{F}$ by linearity. So the real issue is proving that if $\mathcal{F}$ is infeasible then $T\mathcal{F}$ is also infeasible w.o.p. This is where we exploit the fact that $T$ is a random projector. More precisely, we prove the following statements about linear infeasibility w.o.p.:
\begin{enumerate}
\item a nonzero vector is randomly projected to a nonzero vector;
\item if $x$ is not a certificate for $\mathcal{F}$, then it is not a certificate for $T\mathcal{F}$;\label{enum:xnotcert}
\item if $x$ is not a certificate for $\mathcal{F}$ for all $x$ in a finite set $X$, then the same follows for $T\mathcal{F}$;\label{enum:Xnotcert}
\Ky{\item if $b$ is not in the convex hull of $A$, then $Tb$ is not in the convex hull of $TA$.}\label{enum:convfeas}
\item if $b$ is not in the cone of $A$, then $Tb$ is not in the cone of $TA$.\label{enum:mainthm}
\end{enumerate}

The first result is actually a corollary of Lemma \ref{randomprojectionlemma}. We denote by $E^{\mbox{\sf\scriptsize c}}$ the complement of an event $E$.
\begin{corollary}
Let $T$ be a $k\times m$ random projector and $y\in\mathbb{R}^m$ with $y\not=0$. Then we have 
\begin{equation}
\mbox{\sf Prob}(Ty \neq 0) \ge 1 - 2e^{-\mathcal{C}k}. \label{eq:cor1} 
\end{equation}
for some constant $\mathcal{C}>0$ (independent of $n,k$).
\label{cor1}
\end{corollary}
\begin{proof}  {\it Proof.}
For any $\varepsilon\in (0,1)$, we define the following events:
\begin{eqnarray*}
\mathcal{A} & =& \big\{Ty \neq 0 \big \} \\
\mathcal{B} & =& \big \{(1-\varepsilon)\|y\| \le \|Ty\| \le
 (1+\varepsilon)\|y\| \big\}.
\end{eqnarray*}
By Lemma \ref{randomprojectionlemma} it follows that $\mbox{\sf Prob}(\mathcal{B}) \ge 1 - 2e^{-\mathcal{C}\varepsilon^2 k}$ for some constant $\mathcal{C}>0$ independent of $m,k,\varepsilon$.  On the other hand, $\mathcal{A}^{\mbox{\sf\scriptsize c}} \cap \mathcal{B} =\emptyset$, since otherwise, for any $\varepsilon\in(0,1)$ there is a mapping $T_1$ such that $T_1(y) = 0$ and $(1-\varepsilon)\|y\| \le \|T_1(y)\|,$ which altogether imply that $y = 0$ (a contradiction). Therefore, $\mathcal{B}\subseteq\mathcal{A}$, and we have $\mbox{\sf Prob}(\mathcal{A}) \ge \mbox{\sf Prob}(\mathcal{B}) \ge 1 - 2e^{-\mathcal{C}\varepsilon^2 k}$. This holds for all $0 < \varepsilon < 1$, so $\mbox{\sf Prob}(\mathcal{A})\ge 1-2 e^{\mathcal{C}k}$. \qed
\end{proof}  

The following theorem settles
points \ref{enum:xnotcert}-\ref{enum:Xnotcert} above.
\begin{theorem} \label{membershiplemma}
Let $T$ be a $k\times m$ random projector and $\mathcal{F}\equiv\{x\ge 0\;|\;Ax=b\}$ with $A$ an $m\times n$ matrix. Then for any $x\in\mathbb{R}^n$, we have:
\begin{enumerate}[(i)]
\item If $b = \sum\limits_{j=1}^n x_j A_j$ then $Tb=\sum\limits_{j=1}^n x_j TA_j$; \label{lm1}
\item If $b \neq \sum_{j=1}^n x_j A_j$ then $\mbox{\sf Prob}\, \bigg[ Tb \neq \sum_{j=1}^n x_j TA_j\bigg] \ge 1 - 2e^{-\mathcal{C}k}$; \label{lm2}
\item If $b \neq \sum_{j=1}^n x_j A_j$ for all $x \in X \subseteq \mathbb{R}^n$, where $|X|$ is finite, then \[\mbox{\sf Prob}\, \bigg[\forall x\in X\ Tb \neq \sum_{j=1}^n x_j TA_j \bigg] \ge 1 - 2|X|e^{-\mathcal{C}k};\] \label{lm3}
\end{enumerate}
for some constant $\mathcal{C}>0$ (independent of $n,k$).
\end{theorem}
\begin{proof}  {\it Proof.}
 Point \eqref{lm1} follows by linearity of $T$, and \eqref{lm2} by applying Cor.~\ref{cor1} to $Ax-b$. For \eqref{lm3}, the union bound on \eqref{lm2} yields:
\begin{eqnarray*}
\mbox{\sf Prob}\, \bigg[\forall x\in X\ Tb \neq \sum_{j=1}^n x_j TA_j \bigg] & = & \mbox{\sf Prob}\, \bigg[ \bigcap_{x \in X} \; \big\{ Tb \neq \sum_{j=1}^n x_j TA_j\big\}\bigg] \\
= 1 - \mbox{\sf Prob}\, \bigg[ \bigcup_{x \in X} \; \big \{ Tb \neq \sum_{j=1}^n x_j TA_j\big \}^{\mbox{\sf\scriptsize c}}\bigg] & \ge& 1 - \sum_{x \in X} \mbox{\sf Prob}\, \bigg[\big \{ Tb \neq\sum_{j=1}^n x_j TA_j\big \}^{\mbox{\sf\scriptsize c}}\bigg] \\
\mbox{\color{darkgrey} [by \eqref{lm2}]\hspace*{1cm}} & \ge& 1 - \sum_{x \in X} 2e^{-\mathcal{C}k} = 1 - 2|X|e^{-\mathcal{C}k},
\end{eqnarray*}
as claimed. \qed
\end{proof}

Thm.~\ref{membershiplemma} can be used to project certain types of integer programs. It also gives us an indication to why estimating the probability that $Tb \not\in\cone{A}$ is not straightforward. This event can be written as an intersection of uncountably many events $\{Tb\neq \sum_{j=1}^n x_j TA_j\}$ where $x\in\mathbb{R}_+^n$. Even if each of these occurs w.o.p., their intersection might still be small. As these events are dependent, however, we shall show that there is hope yet.

\subsection{Convex hull feasibility}
\Ky{
Next, we show that if the distance between a point  and a closed set is positive, it remains positive with high probability after applying a random projection. We consider the convex hull membership problem: given vectors $b, A_1,\ldots,A_n \in \mathbb{R}^{m}$, decide whether $b\in\mbox{\sf conv}(\{A_1,\ldots,A_n\})$.

We have the following result:
\begin{proposition}
\label{prop:cvxhull}
Given $A_1,\ldots,A_n \in \mathbb{R}^{m}$, let $C=\mbox{\sf conv}(\{A_1,\ldots,A_n\})$, $b \in \mathbb{R}^m$ such that $b \notin C$, $d =
\min\limits_{x \in C} \|b-x\|$ and $D = \max\limits_{1\le j \le n} \|b
-A_j\|$. Let $T:\mathbb{R}^m \to \mathbb{R}^k$ be a random projector. Then
\begin{equation}
  \mbox{\sf Prob} \big [Tb \notin TC \big] \ge 1 - 2n^2 e^{-\mathcal{C}(\varepsilon^2-\varepsilon^3)k}
  \label{eqproj}
\end{equation}
for some constant $\mathcal{C}$ (independent of $m,n,k,d,D$) and
$\varepsilon < \frac{d^2}{D^2}$.
\end{proposition}

\begin{proof}  {\it Proof.}
Let $S_\varepsilon$ be the event that both \[(1-\varepsilon)\|x-y\|^2
\le \|T(x-y)\|^2 \le (1+\varepsilon)\|x-y\|^2\]
 and
\[(1-\varepsilon)\|x+y\|^2 \le \|T(x+y)\|^2 \le
(1+\varepsilon)\|x+y\|^2 \] hold for all $x,y \in \{0, b-A_1,\ldots,b
-A_n\}$. Assume $S_\varepsilon$ occurs. Then for all real $\lambda_j
\ge 0$ with $\sum\limits_{j=1}^n \lambda_j = 1$, we have:
{
\begin{eqnarray}
&& \|Tb- \sum_{j=1}^n \lambda_j TA_j \|^2 
 =  \|\sum_{j=1}^n \lambda_j T(b - A_j) \|^2  \quad \mbox{(by linearity of $T$ and $\sum_j \lambda_j = 1$)} \nonumber \\ [-0.2em]
& = & \sum_{j=1}^n \lambda^2_j \|T(b - A_j) \|^2 + 2 \sum_{1\le i<j\le n} \lambda_i \lambda_j \langle T(b - A_i), T(b-A_j) \rangle \nonumber \\ [-0.2em]
 &=& \sum_{j=1}^n \lambda^2_j \|T(b - A_j) \|^2 + \frac{1}{2}\sum_{1\le i<j\le n} \lambda_i \lambda_j\bigg ( \|T(b - A_i + b - A_j)\|^2 - \|T(A_i-A_j)\|^2 \bigg). \label{key06271}
\end{eqnarray}
Here the last equality follows from the fact that $\langle x, y \rangle = \tfrac{1}{4} (\|x+y\|^2 - \|x -y\|^2)$ for all vectors $x,y$. Moreover, since $S_\varepsilon$ occurs, we have 
$$\|T(b - A_j) \|^2 \ge (1-\varepsilon) \|b - A_j \|^2$$
and 
$$\|T(b - A_i + b - A_j)\|^2 - \|T(A_i-A_j)\|^2 \ge (1 - \varepsilon) \big \|b - A_i + b - A_j \big \|^2  - (1 + \varepsilon) \|A_i - A_j\|^2$$
for all $1 \le i <j \le n$. Therefore, the RHS in (\ref{key06271}) is greater than or equal to
\begin{eqnarray*}
 & & (1 - \varepsilon) \sum_{j=1}^n \lambda^2_j \|b - A_j\|^2 +  \frac{1}{2}\sum_{1\le i<j\le n} \lambda_i \lambda_j \bigg ( (1 - \varepsilon) \big \|b - A_i + b - A_j \big \|^2  - (1 + \varepsilon) \|A_i - A_j\|^2 \bigg)\\ [-0.2em]
 &= & \|b- \sum_{j=1}^n \lambda_j A_j\|^2 
- \varepsilon \bigg(\sum_{j=1}^n \lambda^2_j \|b - A_j\|^2 + \frac{1}{2}\sum_{1\le i<j\le n} \lambda_i \lambda_j ( \|b - A_i + b-A_j\|^2 + \|A_i - A_j\|^2 )\bigg) \\ [-0.2em]
 &=& \|b- \sum_{j=1}^n \lambda_j A_j\|^2 
- \varepsilon \bigg(\sum_{j=1}^n \lambda^2_j \|b - A_j\|^2 + \sum_{1\le i<j\le n} \lambda_i \lambda_j ( \|b - A_i\|^2 + \|b - A_j\|^2 )\bigg).
\end{eqnarray*}
}%
From the definitions of $d$ and $D$, we have $\|b-\sum_{j=1}^n \lambda_j A_j\|^2 \ge d^2$ and $\|b- A_i\| \le D^2$ for all $1\le i \le n$. Therefore:
{
\begin{eqnarray*}
  \|Tb- \sum_{j=1}^n
\lambda_j TA_j \|^2 \ge d^2 - \varepsilon D^2 \bigg(\sum_{j=1}^n
  \lambda^2_j + 2 \sum_{1\le i<j\le n} \lambda_i \lambda_j \bigg) 
= d^2 - \varepsilon D^2 \bigg(\sum_{j=1}^n \lambda_j\bigg)^2 = d^2 -\varepsilon D^2 > 0
\end{eqnarray*}
} due to the fact that $\sum_{j=1}^n \lambda_j = 1$ and the choice of $\varepsilon <
\frac{d^2}{D^2}$. 

Now, since $ \|Tb- \sum\limits_{j=1}^n \lambda_j
TA_j \|^2>0$ for all choices of $\lambda \ge 0$ with  $\sum_{j=1}^n \lambda_j = 1$, it follows that $Tb
\notin\mbox{\sf conv}(\{TA_1, \ldots, TA_n\})$. 

In summary, if
$S_\varepsilon$ occurs, then $Tb \notin \mbox{\sf conv}
(\{TA_1,\ldots, TA_n\})$.  Thus, by Lemma
\ref{randomprojectionlemma2} and the union bound,
{\small
\begin{equation*}
\mbox{\sf Prob} (Tb \notin TC) \ge \mbox{\sf Prob} (S_\varepsilon)
\ge 1 - 2\big(n + 2{\scriptsize {n \choose 2}}\big)
e^{-\mathcal{C}(\varepsilon^2-\varepsilon^3) k} = 1 -2
n^2e^{-\mathcal{C}(\varepsilon^2-\varepsilon^3) k}
\end{equation*}
}
for some constant $\mathcal{C} > 0$. \qed
\end{proof}

As an interesting aside, we remark that this proof can also be extended to show that disjoint polytopes project to disjoint polytopes with high probability.
}

\subsection{Cone feasibility}
We now deal with the last (and most relevant) result: if $b$ is not in the cone of the columns of $A$, then $Tb$ is not in the cone of the columns of $TA$ w.o.p. We first define the $A$-norm of $x\in\cone{A}$ as
\[\|x\|_A = \min\big\{\sum\limits_{j=1}^{n} \lambda_j \;\big|\;\lambda \ge 0\land x =\sum\limits_{j=1}^{n} \lambda_j A_j \big\}.\]
For each $x\in\cone{A}$, we say that $\lambda\in\mathbb{R}^n_+$ yields a \emph{minimal $A$-representation} of $x$ if and only if $\sum\limits_{j=1}^{n} \lambda_j = \|x\|_A$. We define $\mu_A = \max \{\|x\|_A \;|\; x \in\cone{A}\land\|x\|\le 1\}$; then, for all $x\in\cone{A}$, we have
\[\|x\| \le \|x\|_A \le \mu_A \|x\|.\]
In particular $\mu_A \ge 1$. Note that $\mu_A$ serves as a measure of worst-case distortion when we move from Euclidean to $\|\cdot\|_A$ norm.

For the next result, we assume we are given an estimate of a lower bound $\Delta$ to $d=\min\limits_{x \in C} \|b-x\|$, and also (without loss of generality) that $b$ and the column vectors of $A$ have unit Euclidean norm.
\begin{theorem}
Given an $m\times n$ matrix $A$ and $b\in\mathbb{R}^m$ s.t.~$b\not\in\cone{A}$.  Then for any $0<\varepsilon \Ky{<}\frac{\Delta^2}{\mu_A^2+2\mu_A\sqrt{1-\Delta^2}+1}$ \Ky{and any $k\times m$ random projector $T$ (such as one in Section 2),} we have
\begin{equation}
  \mbox{\sf Prob}(Tb \notin \cone{TA}) \ge 1 - 2\Ky{(n+1)(n+2)}e^{-\mathcal{C}(\varepsilon^2-\varepsilon^3)k}
  \label{eqnormAapp}
\end{equation}
for some constant $\mathcal{C}$ (independent of $m,n,k,\Delta$).
\label{thm:mainthm}
\end{theorem}
\begin{proof}  {\it Proof.}
For any $\varepsilon$ chosen as in the theorem statement, let $S_\varepsilon$ be the event that both 
$$(1-\varepsilon)\|x-y\|^2 \le \|T(x-y)\|^2 \le (1+\varepsilon)\|x-y\|^2$$ and $$(1-\varepsilon)\|x+y\|^2 \le \|T(x+y)\|^2 \le (1+\varepsilon)\|x+y\|^2$$
 hold for all $x,y \in \{\Ky{0}, b,A_1,\ldots,A_n\}$. By Lemma \ref{randomprojectionlemma}, we have
$$\mbox{\sf Prob} (S_\varepsilon) \ge 1 - 4{\Ky{n+2} \choose
  2}e^{-\mathcal{C}(\varepsilon^2-\varepsilon^3) k} = 1 -
2\Ky{(n+1)(n+2)}e^{-\mathcal{C}(\varepsilon^2-\varepsilon^3) k}$$ for some
constant $\mathcal{C}$ (independent of $m,n,k,d$). We will prove that
if $S_\varepsilon$ occurs, then we have $Tb \notin \mbox{\sf cone}
\{TA_1,\ldots, TA_n\}$. Assume that $S_\varepsilon$
occurs. Consider an arbitrary $x \in \mbox{\sf
  cone}\{A_1,\ldots,A_n\}$ and let $\sum\limits_{j=1}^{n} \lambda_j
A_j$ be a minimal $A$-representation of $x$. Then we have:
{\small
\begin{eqnarray}
&&\|Tb - Tx\|^2 =  \|Tb - \sum_{j=1}^n \lambda_j TA_j \|^2 \nonumber \\ [-0.2em]
&= & \|Tb\|^2 + \sum_{j=1}^n \lambda_j^2 \|TA_j\|^2 - 2\sum_{j=1}^n \lambda_j \langle Tb ,TA_j \rangle +   2\sum_{1 \le i < j \le n} \lambda_i\lambda_i \langle TA_i ,TA_j \rangle \nonumber \\ [-0.2em]
&=&\!\!\|Tb\|^2\!\!+\!\!\sum_{j=1}^n\!\lambda_j^2 \|TA_j\|^2\!\!+\!\!\sum_{j=1}^n \frac{\lambda_j}{2} (\|T(b\!-\!A_j)\|^2\!\!-\!\|T(b\!+\!A_j)\|^2 )\!+\!\!\!\!\!\!\sum_{1 \le i < j \le n}\!\!\!\!\!\!\frac{\lambda_i\lambda_j}{2}(\|T(A_i\!+\!A_j)\|^2\!\!-\!\|T(A_i\!-\!A_j)\|^2) \nonumber \\
\label{key205612}
\end{eqnarray}
}%
\Ky{Here the last equality follows by the fact that $\langle x, y \rangle = \tfrac{1}{4} (\|x+y\|^2 - \|x -y\|^2)$ for all vectors $x,y$.} Moreover, since $S_\varepsilon$ occurs, \Ky{we have 
$$\|Tb\|^2 \ge (1-\varepsilon) \|b\|^2, \qquad \|TA_j\|^2 \ge (1-\varepsilon) \|A_j\|^2 \quad \mbox{for all } 1 \le j \le n$$
and 
\begin{eqnarray*}
&&\|T(b - A_j)\|^2 - \|T(b + A_j)\|^2 \ge (1 - \varepsilon) \big \|b - A_j \big \|^2  - (1 + \varepsilon) \|b + A_j\|^2 \\
&&\|T(A_i\!+\!A_j)\|^2\!\!-\!\|T(A_i\!-\!A_j)\|^2  \ge (1 - \varepsilon) \|A_i\!+\!A_j\|^2\!\! - (1 + \varepsilon) \!\|A_i\!-\!A_j\|^2
\end{eqnarray*}
for all $1 \le i <j \le n$. Therefore, the RHS in (\ref{key205612}) is greater than or equal to}
\begin{eqnarray}
&& (1 - \varepsilon) \|b\|^2 + (1 - \varepsilon) \sum_{j=1}^n \lambda_j^2 \|A_j\|^2 + \sum_{j=1}^n \frac{\lambda_j}{2} ((1 - \varepsilon)\|b - A_j\|^2 - (1 + \varepsilon) \|b+A_j\|^2 ) \nonumber  \\ 
&& \hspace*{4cm} + \sum_{1 \le i < j \le n}\frac{\lambda_i\lambda_j}{2} ((1 - \varepsilon) \|A_i + A_j\|^2 - (1 + \varepsilon) \|A_i-A_j\|^2 ).
\end{eqnarray}
\Ky{Since we have assumed that 
$\|b\| = \|A_1\| = \ldots \|A_n\| = 1$, it can then} be rewritten as
\begin{eqnarray*}
&&  \|b - \sum_{j=1}^n \lambda_j A_j \|^2 
 - \varepsilon \bigg (1 + \sum_{j=1}^n \lambda_j^2 + 2 \sum_{i=j}^n \lambda_j + 2 \sum_{j \neq i} \lambda_i \lambda_j \bigg) \\
&= & \|b - \sum_{j=1}^n \lambda_j A_j \|^2 
- \varepsilon \big (1 + \sum_{j=1}^n \lambda_j \big)^2  \\
&= & \|b - x\|^2 
- \varepsilon \big (1 + \|x\|_A \big)^2  \qquad \mbox{\Ky{(by the definition of $A$-norm).}}
\end{eqnarray*}
\Ky{In summary, we have proved that, when the event $S_\varepsilon$ occurs, then 
\begin{equation}  \label{eq:7202}
\|Tb - Tx\| \ge \|b - x\|^2 
- \varepsilon \big (1 + \|x\|_A \big)^2.
\end{equation}
}

Denote by $\alpha=\|x\|$ and let $p$ be the \Ky{orthogonal} projection of $b$ to
$\mbox{\sf cone}\{A_1,\ldots,A_n\}$, which \Ky{means} $\|b-p\|=\min
\{\|b-x\|\;|\;x \in \mbox{\sf cone}\{A_1,\ldots,A_n\}\}$. We will need to use the following claim:
\begin{quote}
  {\bf Claim}. For all $b,x,\alpha,p$ given above, we have
  $\|b-x\|^2\ge\alpha^2-2 \alpha\|p\|+1$.
\end{quote}
By this claim (proved later), \Ky{from inequality (\ref{eq:7202})}, we have:
\begin{eqnarray*}
 \|Tb - Tx\|^2  & \Ky{\ge}& \alpha^2 - 2\alpha\|p\| + 1 - \varepsilon
 \big (1 + \|x\|_A \big)^2 \\ 
& \ge & \alpha^2 - 2\alpha\|p\| + 1 - \varepsilon \big (1 + \mu_A
 \alpha\big)^2 \quad \mbox{\Ky{(since $  \|x\|_A \le \mu_A \|x\|$)} }\\
 & = & \big(1 - \varepsilon \mu_A^2\big) \alpha^2 - 2 \big( \|p\| +
 \varepsilon \mu_A\big)\alpha + (1-\varepsilon). 
 \end{eqnarray*}
The last expression can be viewed as a quadratic function with respect
to $\alpha$. We will prove this function is \Ky{positive} for all
$\alpha\in\mathbb{R}$. This is equivalent to\footnote{\Ky{Here we use the fact that a quadratic function $ax^2 + bx + c > 0$ for all $x \in \mathbb{R}$ if and only if $a > 0$ and $b^2 - 4ac < 0$.}}
\begin{eqnarray*} 
&&\big(\|p\| + \varepsilon \mu_A\big)^2 -  \big(1 - \varepsilon \mu_A^2 \big) (1-\varepsilon) \; \Ky{<} \; 0 \\
&\Leftrightarrow&  \big(\mu_A^2 +  2 \|p\|\mu_A + 1\big)\varepsilon  \; \Ky{<} \;  1 - \|p\|^2  \\
& \Leftrightarrow & \varepsilon  \; \Ky{<} \;  \frac{1 - \|p\|^2}{\mu_A^2  + 2 \|p\|\mu_A + 1} =  \frac{d^2}{\mu_A^2  + 2 \|p\|\mu_A + 1},
\end{eqnarray*}
which holds for the choice of $\varepsilon$ as in the hypothesis.  In
\Ky{conclusion}, if the event $S_\varepsilon$ occurs, then $\|Tb - Tx\|^2
> 0$ for all $x \in \mbox{\sf cone}\{A_1,\ldots,A_n\}$, i.e.  $Tx
\notin \mbox{\sf cone}\{TA_1,\ldots,TA_n\}$. Thus,
\[\mbox{\sf Prob}(Tb \notin TC) \ge \mbox{\sf Prob} (S_\varepsilon) \ge 1 - 2\Ky{(n+1)(n+2)}e^{-c(\varepsilon^2 - \varepsilon^3)k}\]
as claimed. The result follows since $\|p\|_2^2+d^2=1$ by Pythagoras' theorem, and $\Delta\le d$. 
\par\bigskip\par
{
\noindent {\it Proof of the claim \Ky{that  $\|b-x\|^2\ge\alpha^2-2 \alpha\|p\|+1$:}}\\
If $x = 0$ then the claim is
trivially true, since $\|b-x\|^2= \|b
\|^2=1=\alpha^2-2\alpha\|p\|+1$. Hence we assume $x\neq 0$.  First
consider the case $p\neq 0$. By Pythagoras' theorem, we must have $d^2
= 1 - \|p\|^2$.  We denote by $z=\frac{\|p\|}{\alpha}x$, then $\|z\| =
\|p\|$. Set $\delta = \frac{\alpha}{\|p\|}$, we have
\begin{eqnarray*}
\|b - x \|^2  & =& \|b - \delta z \|^2 \\
& =& (1 - \delta) \|b\|^2 + (\delta^2 - \delta) \|z\|^2 + \delta \|b - z\|^2 \\
& =& (1 - \delta) + (\delta^2 - \delta) \|p\|^2 + \delta \|b - z\|^2\\
& \ge& (1 - \delta) + (\delta^2 - \delta) \|p\|^2 + \delta d^2\\
& =& (1 - \delta) + (\delta^2 - \delta) \|p\|^2 + \delta (1 - \|p\|^2) \\
& =& \delta^2 \|p\|^2 - 2\delta \|p\|^2 + 1  = \alpha^2 - 2\alpha\|p\| + 1.
\end{eqnarray*}
Next, we consider the case $p = 0$. In this case we have $b^T(x) \le 0$
for all $x \in \mbox{\sf cone}\{A_1,\ldots,A_n\}$. Indeed, for an arbitrary $\delta >0$,
\[0 \le \frac{1}{\delta} (\|b-\delta x\|^2 - 1) = \frac{1}{\delta} (1 + \delta^2 \|x\|^2 - 2 \delta b^Tx - 1) = \delta \|x\|^2 - 2 b^Tx\]
which tends to $-2 b^Tx$ when $\delta \to 0^+$. Therefore
\begin{eqnarray*}
\|b - x \|^2 & =& 1 - 2b^Tx + \|x\|^2 \ge \|x \|^2 + 1 = \alpha^2 -
2\alpha\|p\| + 1,
\end{eqnarray*}
which proves the claim.}  \qed
\end{proof}

Since cone membership is the same as LP feasibility, Thm.~\ref{thm:mainthm} establishes that LFPs can be randomly projected accurately w.o.p.

\section{Preserving optimality}
\label{s:opt}
In this section we show that, if the projected dimension $k$ is large enough, $v(P)\approx v(P_T)$ w.o.p (Thm.~\ref{thm:objfunapprox}). We assume all along, and without loss of generality, that $b,c$ and the columns of $A$ have unit Euclidean norms.

The proof of Thm.~\ref{thm:objfunapprox} is divided into two main parts.
\begin{itemize}
\item In the first part, we write $v(P)\approx v(P_T)$ formally as ``given $\delta>0$ there is a random projector $T$ such that $v(P)-\delta\le v(P_T)\le v(P)$ w.o.p.'', formalize some infeasible LFPs which encode $v(P)-\delta$ and $v(P_T)$, and emphasize their relationship.
\item In the second part, we formally argue the ``overwhelming probability'' by means of an $\varepsilon>0$ (in function of $\delta$) which ensures that the probability of $v(P)-\delta\le v(P_T)$ approaches 1 fast enough (as a function of $\varepsilon$). This $\varepsilon$ refers to the projected (infeasible) LFP of the first part, but for technical reasons we cannot simply ``inherit it'' from Thm.~\ref{thm:mainthm}. Instead, from the cone of the infeasible LFP we carefully construct a new {\it pointed} cone which allows us to carry out a projected separation argument based on inner product preservation (Prop.~\ref{cor:angles}).
\end{itemize}
Our proof assumes that the feasible region of $P$ is non-empty and bounded. Specifically, we assume that a constant $\theta>0$ is given such that that there exists an optimal solution $x^*$ of $P$ (see Eq.~\eqref{orig}) satisfying
\begin{equation}
\sum\limits_{j=1}^n x^*_j<\theta. \label{eq:hypopt}
\end{equation}
For the sake of simplicity (and without loss of generality), we assume further that $\theta \ge 1$. This assumption is used to control the excessive flatness of the involved cones, which is required in the projected separation argument.


\subsection{A cone transformation operation}
Before introducing Thm.~\ref{thm:objfunapprox} and its proof, we explain how to construct a pointed cone from the cone of the LFP in such a way as to preserve a certain membership property.

Given a polyhedral cone
\[\mathcal{K}=\left\{\sum\limits_{j\le n} x_jC_j\;\bigg|\;x\in\mathbb{R}^n_+\right\}\]
in which $C_1,\ldots,C_n$ are column vectors of an $m\times n$
matrix $C$, \Ky{in other words $\mathcal{K} = \cone{C}$}. For any $u\in\mathbb{R}^m$, we
consider the following transformation $\phi_{u,\theta}$, defined by:
\[\phi_{u,\theta}(\mathcal{K}) :=\left\{\sum\limits_{j=1}^n \Ky{x_j \left(C_j-\frac{1}{\theta}u\right)\;} \bigg| \;  x \in \mathbb{R}^n_+ \right\}.\]
In particular, $\phi_{u,\theta}$ moves the origin in the
direction $u$ by a step $1/\theta$ (see Figure \ref{fig:M2}). For $\theta$ defined in Eq.~\eqref{eq:hypopt}, we also consider the following set
\[\mathcal{K}_\theta = \left\{\sum\limits_{j=1}^n \, \Ky{x_jC_j} \; \bigg| \;  x \in \mathbb{R}^n_+\land \sum\limits_{j=1}^nx_j<\theta\right\}.\]
$\mathcal{K}_\theta$ can be seen as a set truncated from $\mathcal{K}$ \Ky{(in particular, it is not a cone anymore)}.  We shall show
that $\phi_{u,\theta}$ preserves the membership of the vector $u$ in the ``truncated
cone'' $\mathcal{K}_\theta$. 
\begin{figure}
\ifmor
\else
  \centering{\includegraphics[width=10cm,height=7cm]{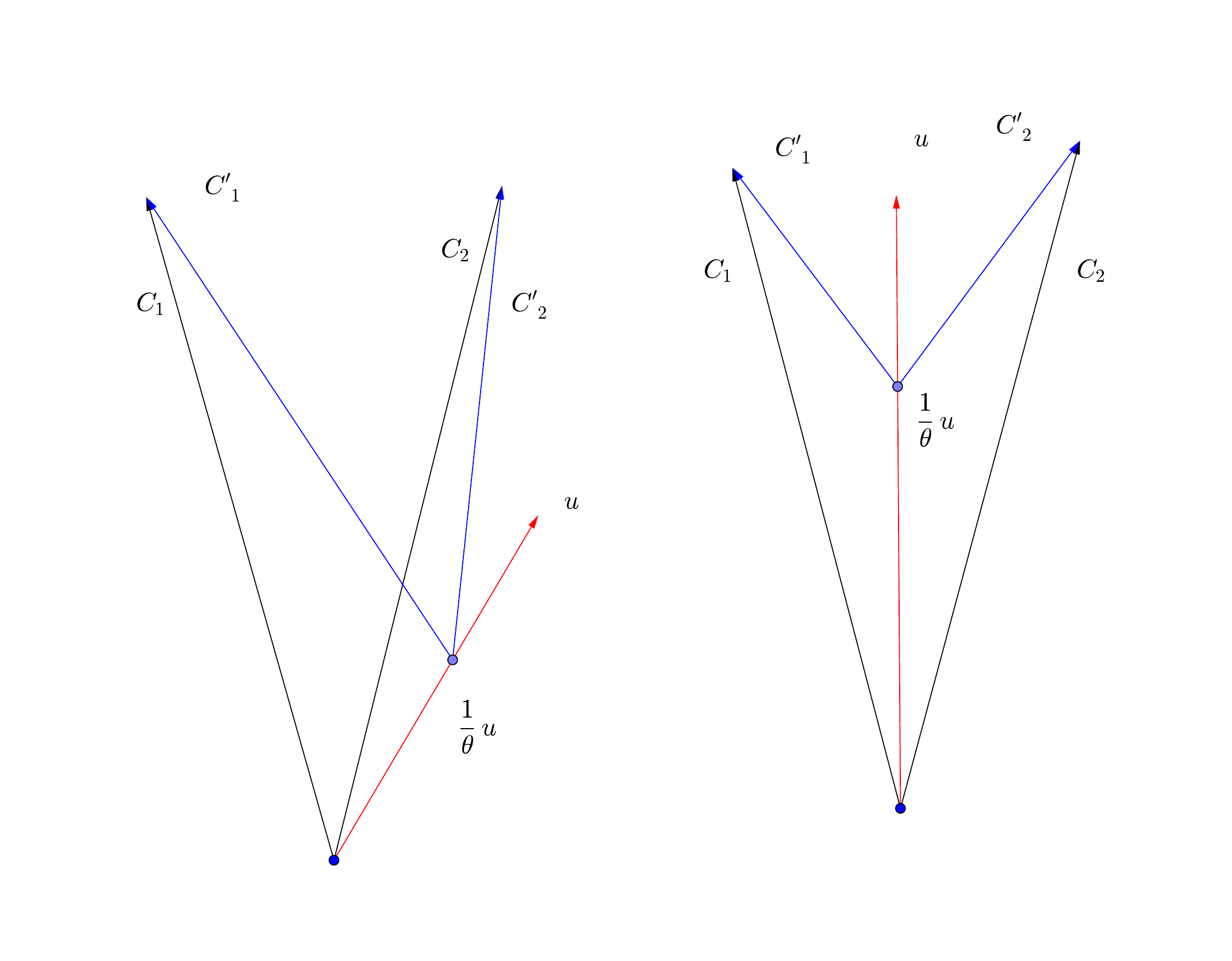}}
\\  \centerline{\small The effect of $\phi_u$ when $u$ does not belong to the cone (left) and when it does (right).\label{fig:M2}}
\fi      
\end{figure}

\begin{lemma}\label{lem:conictransform}
For any $u \in \mathbb{R}^m$, we have  $u \in \mathcal{K}_\theta$
if and only if $u \in \phi_{u,\theta}(\mathcal{K})$.
\end{lemma}   
\begin{proof}  {\it Proof.}
\Ky{First of all, let denote by $t = 1-\frac{1}{\theta}\sum\limits_{j=1}^nx_j$. }

($\Rightarrow$) If $u \in \mathcal{K}_\theta$, then there
  exists $x \in \mathbb{R}^n_+$ such that
  $u=\sum\limits_{j=1}^nx_j C_j$ and
  $\sum\limits_{j=1}^nx_j<\theta$. \Ky{Then $u$
  can be written as $\sum\limits_{j=1}^n x'_j \big(C_j -\frac{1}{\theta}u \big)$ with
  $x'=\frac{1}{t} \, x$. Indeed,
  \begin{align*}
  \sum\limits_{j=1}^n x'_j \big(C_j -\frac{1}{\theta}u \big) & = \frac{1}{t}  \sum\limits_{j=1}^n x_j \big(C_j -\frac{1}{\theta}u \big) \\
  & = \frac{1}{t}  \sum\limits_{j=1}^n x_j C_j - \frac{1}{t}   \big(\sum\limits_{j=1}^n  \frac{1}{\theta} x_j \big) u \\
  & = \frac{1}{t}  u - \frac{1}{t}   \big(\sum\limits_{j=1}^n  \frac{1}{\theta} x_j \big) u \\
  & = \frac{1}{t}  \big( 1 -  \frac{1}{\theta}   \sum\limits_{j=1}^n  x_j \big)  u \\
  & = u \qquad \mbox{(by definition of $t$)}.
  \end{align*}
}
  
\Ky{Moreover}, due to the assumption that $\sum\limits_{j=1}^nx_j<\theta$, \Ky{we have} $x' \ge 0$. \Ky{It follows that $u\in\phi_{u,\theta}(\mathcal{K})$.} \\
  
($\Leftarrow$) If $u \in \phi_{u,\theta}(\mathcal{K})$, then there
  exists $x \in \mathbb{R}_+^n$ such that
  $u=\sum\limits_{j=1}^n x_j \Ky{\big(C_j -\frac{1}{\theta}u \big)}$. \Ky{It is equivalent to $\big(1+\frac{1}{\theta}\sum\limits_{j=1}^nx_j ) u = \sum\limits_{j=1}^n x_j C_j.$}
  \Ky{Thus} $u$ can also be written as
  $\sum\limits_{j=1}^n x'_j \Ky{C_j}$, where
  $x'_j=\frac{x_j}{1+\frac{1}{\theta}\sum\limits_{i=1}^nx_i}$. Note
  that $\sum\limits_{j=1}^nx'_j < \theta$ because
  \[ \sum\limits_{j=1}^nx'_j =\frac{\sum\limits_{j=1}^nx_j}{1+\frac{1}{\theta}\sum\limits_{j=1}^nx_j}< \theta, \]
  which implies that $u \in K_\theta$. \qed
\end{proof}

Note that this result is still valid when the transformation $\phi_{u,\theta}$
is only applied to a subset of columns of $C$. Given \Ky{any vector $u$ and} an index set $J\subseteq\{1,\ldots,n\}$, we define $\forall j\le n$:
\[ \Ky{C_j^{Ju}} = \left\{\begin{array}{ll}
C_j - \frac{1}{\theta}u & \quad \mbox{if } j\in J \\ C_j & \quad
\mbox{otherwise.} \end{array}\right.\]
We extend $\phi_{u,\theta}$ to
\begin{equation}
\phi_{u,\theta}^J(\mathcal{K})=\left\{\sum\limits_{j=1}^n \, \Ky{x_jC_j^{Ju}} \;\bigg| \;  x \in \mathbb{R}^n_+ \right\}
= \Ky{\cone{C_j^{Ju} \;|\; 1 \le j \le n}}, \label{eq:phiext}
\end{equation}
and define \[\mathcal{K}^J_\theta = \left\{\sum\limits_{j=1}^n\, \Ky{x_jC_j} \; \bigg| \;  x \in \mathbb{R}^n_+\land \sum\limits_{j \in J}x_j<\theta\right\}
.\]
The following corollary \Ky{can be proved} in the same way as Lemma \ref{lem:conictransform}, in which $\phi_{u,\theta}$ is replaced by $\phi_{u,\theta}^J$.
\begin{corollary}\label{cor:conictransform}
For any \Ky{vector} $u \in \mathbb{R}^m$ and \Ky{any index set} $J\subseteq\{1,\ldots,n\}$, we have
 $u \in \mathcal{K}^J_\theta$ if and only if $u \in
\phi_{u,\theta}^J(\mathcal{K})$.
\end{corollary}   

\subsection{The main theorem}
Given an LFP instance $Ax=b\land x\ge 0$, where $A$ is an $m\times n$ matrix and
$T$ is a $k\times m$ random projector.  By Thm.~\ref{thm:mainthm}, we know that, 
\[\exists x\ge 0\;(Ax=b)\quad\Leftrightarrow\quad \exists x\ge 0\;(TAx=Tb)\]
w.o.p. We remark that this also holds for a $(k+h)\times m$
random projector of the form
\[\left(\begin{array}{c} I_h\quad 0 \\ T\end{array}\right),\]
where $T$ is a $k\times m$ random matrix. This allows us to claim
the feasibility equivalence w.o.p.~even when we only want to project a
subset of rows of $A$. In the following, we will use this observation to handle constraints
and objective function separately. In particular, we only project the constraints while keeping objective function unchanged.


If we add the constraint $\sum\limits_{j=1}^nx_j \le\theta$ to the problem $P_T$ (defined in Eq.~\eqref{proj}), we obtain the following:
 \begin{equation}
 P_{T,\theta}\equiv\min\left\{\transpose{c}x\;\bigg|\;TAx=Tb\land \sum\limits_{j=1}^nx_j \le \theta \land  x\in\mathbb{R}^n_+\right\}. \label{proj2}
 \end{equation}

So we come to our main theorem, which asserts that the optimal objective value of $P$ can be well-approximated by that of $P_{T,\theta}$.
\begin{theorem}
Assume $\mathcal{F}(P)$ is bounded and non-empty. Let $y^*$ be an optimal dual solution of $P$ of minimal Euclidean norm. Given $\Ky{0 < \delta \le |v(P)|}$, we have
\begin{equation}
  v(P)-\delta\le v(P_{T,\theta})\le v(P),\label{eq:objthm}
\end{equation}
with probability at least $p= 1 - 4ne^{-\mathcal{C}(\varepsilon^2-\varepsilon^3) k}$, where \Ky{ $\varepsilon=O(\frac{\delta}{\theta^2 \|y^*\|})$}.
\label{thm:objfunapprox}
\end{theorem}
First, we will informally explain the idea of the proof. Since $v(P)$ is the optimal objective value of problem $P$, for any positive $\delta$, the problem 
$$Ax = b \land x \ge 0 \land c^\top  x \le v(P) - \delta.$$
is infeasible \Ky{(because we can not obtain a lower objective value than $v(P)$).} That problem can now be projected in such a way that it remains infeasible w.o.p. By rewriting this original problem in the standard form as
\begin{equation} \label{infeasible-problemZQW}
\begin{pmatrix} c^\top & 1 \\ A & 0
\end{pmatrix} \begin{pmatrix} x \\ s
\end{pmatrix}  = \begin{pmatrix} v(P) - \delta \\ b
\end{pmatrix}, \mbox{ where} \begin{pmatrix} x \\ s
\end{pmatrix} \ge 0,
\end{equation}
and applying a random projection of the form
$$\Ky{\left (
\begin{array}{c|ccc}
	1 & 0 & \ldots & 0  \\
	\hline
	0& && \\
	\ldots & & T & \\
	0& && 
\end{array}
\right)
,} \mbox{ where $T$ is a $k \times m$ random projector}, 
$$
we will obtain the following problem, which is supposed to be infeasible w.o.p.
\begin{equation}
\left.\begin{array}{rcl}
cx+s &=& v(P)-\delta\\
TAx &=& T b\\
s &\ge& 0 \\
x &\ge& 0 \end{array}\right\}\label{eq:Tsys}.
\end{equation}
The main idea is that, the prior information about the optimal solution $x^*$ (i.e. \Ky{the condition} $\sum\limits_{j=1}^n x^*_j \le \theta$), can now be added into this new projected problem. \Ky{This does not change its
feasibility}, but later can be used to transform the corresponding cone into the one which is easier to deal with. Therefore, w.o.p., the problem
\begin{equation}
\left.\begin{array}{rcl}
cx & \le& v(P)-\delta\\
T Ax &=& T b\\
\sum\limits_{j=1}^n x_j &\le& \theta \\
x &\ge& 0 \end{array}\right\}\label{eq:Tsys2}
\end{equation}
 is infeasible. Hence we deduce that $cx\ge v(P)-\delta$ holds
w.o.p.~for any feasible solution $x$ of the problem $P_{T,\theta}$, and that proves
the LHS of Eq.~\eqref{eq:objthm}. For the RHS, the proof is trivial since $P_T$ is a relaxation of $P$ with the same objective function. We now turn to the formal proof.

\begin{proof}  {\it Proof.}
Let $$\tilde{A}=\begin{pmatrix} c^\top & 1 \\ A & 0
\end{pmatrix}, \tilde{x} = \begin{pmatrix} x \\ s
\end{pmatrix}  \mbox{ and } \tilde{b} =\begin{pmatrix} v(P) - \delta \\ b
\end{pmatrix}$$
Furthermore, let $$ \Ky{\tilde{T} = \left (
\begin{array}{c|ccc}
1 & 0 & \ldots & 0  \\
\hline
0& && \\
\ldots & & T & \\
0& && 
\end{array}
\right)
, \mbox{ where $T$ is a $k \times m$ random projector}.}
$$
In the rest of the proof, we prove that
  $\tilde{b}\not\in\cone{\tilde{A}}$ if and only if
  $T\tilde{b}\not\in\cone{T\tilde{A}}$ w.o.p.

Let $J$ be the index set of
the first $n$ columns of $\tilde{A}$. Consider the transformation
$\phi_{\tilde{b},\theta'}^J$ as defined above, using a step
$\frac{1}{\theta'}$ instead of $\frac{1}{\theta}$, in which 
$\theta'\in(\theta,\theta+1)$. 
We define the following matrix:
\[A'=\begin{pmatrix}
\tilde{A}_1 - \frac{1}{\theta'}\tilde{b} & \cdots & \tilde{A}_n -
\frac{1}{\theta'}\tilde{b} & \tilde{A}_{n+1}
\end{pmatrix}\]

Since Eq.~\eqref{infeasible-problemZQW} is infeasible, it is easy to verify that the system:
\begin{equation}
  \left.\begin{array}{rcl}
    \tilde{A}\tilde{x} &=& \tilde{b} \\
    \sum\limits_{j=1}^n \tilde{x}_j &<& \theta' \\
    \tilde{x} &\ge& 0
  \end{array}\right\}
  \label{eq:tilde1}
\end{equation}
is also infeasible. \Ky{It is equivalent to
\[ \tilde{b}\not\in \left\{\sum\limits_{j=1}^n\, \tilde{x}_j\tilde{A}_j \; \bigg| \;  \tilde{x} \in \mathbb{R}^n_+\land \sum\limits_{j \in J}\tilde{x}_j<\theta'\right\}.
\]
}Then, by Cor.~\ref{cor:conictransform}, it follows that
$\tilde{b}\not\in\cone{A'}$. 

Let $y^*\in \mathbb{R}^m$ be an optimal dual solution of
$P$ of minimal Euclidean norm. By the strong duality theorem, we have $y^*\,A \le c$ and $y^*\,b=v(P)$. We define \Ky{
$$\tilde{y}=\begin{pmatrix} 1 \\ -y^\ast
\end{pmatrix}.$$}
We will prove that $\tilde{y}\,A' >0$ and $\tilde{y}\,
\tilde{b}<0$. 
Indeed, since  $\tilde{y}\,
\tilde{A}= \Ky{\begin{pmatrix} 1 \\ -y^\ast
\end{pmatrix}^\top \begin{pmatrix} c^\top & 1 \\ A & 0
\end{pmatrix}  =} \begin{pmatrix} c-y^{*}\,A\\ 1
\end{pmatrix}\ge 0$ and 
$\tilde{y}\,\tilde{b}=v(P)-\delta- \Ky{y^*\,b}= -\delta < 0$, then we have 
\begin{equation}
\label{farkasvec}
 \tilde{y}\, A' = \begin{pmatrix}
   c-y^*\,A+\frac{\delta}{\theta'}\\ 1
\end{pmatrix} \ge \frac{\delta}{\theta'}\mathbf{1}\ge \frac{\delta}{\theta+1}\mathbf{1} \mbox{ and  } \tilde{y}\,\tilde{b} = -\delta
\end{equation}
(where $\mathbf{1}$ is the all-one vector), which proves the claim.

Now we can apply the scalar product preservation property. By Proposition \ref{cor:angles} and the union bound, we have that
\begin{align}
\label{eq:angle}
\LL{\forall j\le n} \quad |\,((\tilde{T}\tilde{y})\,(\tilde{T}A')-\tilde{y}\,A')_j\,| & \le \varepsilon \eta \\
 \Ky{|\,(\tilde{T}\tilde{y})\,(\tilde{T}\tilde{b})-\tilde{y}\,\tilde{b}\,|} & \le \varepsilon \eta \label{eq:angle2}
\end{align}
hold with probability at least $p=1-4ne^{-\mathcal{C}(\varepsilon^2-\varepsilon^3) k}$. Here, $\eta$ is the normalization constant \Ky{(to scale vectors to unit norm)}
\[\eta = \max \bigg\{\|\tilde{y}\| \,\|\tilde{b}\|, \; \max\limits_{1\le j \le
  n} \; \|\tilde{y}\| \,\|{A_j}'\| \bigg\},\]
in which \Ky{we can easily estimate} $\eta = O(\theta \|y^*\|)$ (the proof is given at the end).
Let us now fix
$\varepsilon=\frac{\delta}{2(\theta+1)\eta}$.  \Ky{It is easy to see that 
\[\varepsilon=\frac{\delta}{2(\theta+1)\eta} = O(\frac{\delta}{\theta^2 \|y^*\|}).
\]

Then with this choice of $\varepsilon$}, by \eqref{farkasvec},
\eqref{eq:angle} and \eqref{eq:angle2}, we have, with probability at least $p$,  
\begin{eqnarray*}
  (\tilde{T}\tilde{y})\,(\tilde{T}A') &\ge& \Ky{\tilde{y}\,A' -  \varepsilon \eta\mathbf{1}  \ge \left(\frac{\delta}{\theta+1}  -  \varepsilon \eta\right)\mathbf{1}} \ge 0 \\
  (\tilde{T}\tilde{y})\,(\tilde{T}\tilde{b}) & \le & \Ky{\tilde{y} \tilde{b}+ \varepsilon \eta \le -\delta +  \varepsilon \eta} < 0,
\end{eqnarray*}
which then implies that the problem
\begin{eqnarray*}
	\tilde{T}A'\tilde{x} &=& \tilde{T}\tilde{b} \\
	\tilde{x} &\ge& 0
\end{eqnarray*}
is infeasible (by Farkas' Lemma). By definition, 
$\tilde{T}A'\tilde{x}=\tilde{T}\tilde{A}\tilde{x}-\frac{1}{\theta'}\sum\limits_{j=1}^nx_j\tilde{T}\tilde{b}$,
which implies that \Ky{the system}
\begin{equation*}
  \left.\begin{array}{rcl}
    \tilde{T}\tilde{A}\tilde{x} &=& \tilde{T}\tilde{b} \\
    \sum\limits_{j=1}^n\tilde{x}_j &<& \theta' \\
    \tilde{x} & \ge& 0
    \end{array}\right\}
\end{equation*}
is also infeasible with probability at least $p$ (the proof is similar to that of Corollary \ref{cor:conictransform}).  Therefore, with probability at least $p$, the following optimization problem:
\begin{equation*}
\inf\left\{\transpose{c}x\;\bigg|\;TAx=Tb\land \sum\limits_{j=1}^nx_j < \theta' \land  x\in\mathbb{R}^n_+\right\}. 
\end{equation*}
has its optimal value greater than $v(P)-\delta$. Since $\theta'>\theta$, it follows that with probability at least $p$, \Ky{we have} $v(P_{T,\theta}) \ge v(P)-\delta$, as claimed. \Ky{The proof is done}.

\textbf{Proof of the claim that $\eta = O(\theta \|y^*\|)$:} 
We have 
\begin{align*}
	\|\tilde{b}\|^2 & = \|b\|^2 + (v(P) - \delta)^2  \qquad \Ky{\mbox{(by the definition of $\tilde{b}$)}}\\
	& \le \Ky{\|b\|^2 + 2 (v(P))^2 + 2 \delta^2 \qquad \mbox{(using the inequality $(x-y)^2 \le 2x^2 + 2y^2$ for all $x,y$.)}}\\
	& \le \|b\|^2 + \Ky{4 (v(P))^2 \qquad \mbox{(by assumption that $|\delta| \le |v(P)|$)}} \\
	& = 1 + \Ky{4} |c^\top x^*| \\
	&\le 1 + \Ky{4} \|c\|_\infty \, \|x^*\|_1 \quad \mbox{(by H\"{o}lder inequality)} \\
	& \le 1 + \Ky{4} \theta \quad \mbox{(since $\|c\|_\infty \le \|c\|_2 = 1$  and $\sum x^*_i \le \theta$)} \\
	& \le \Ky{5} \theta \quad \mbox{(by the assumption that $ \theta \ge 1$).}
\end{align*}
Therefore, we conclude that
\[\eta = \max \bigg\{\|\tilde{y}\| \,\|\tilde{b}\|, \; \max\limits_{1\le j \le
	n} \; \|\tilde{y}\| \,\|{A_j}'\| \bigg\}  = O(\theta \,\|y^*\|)\] \qed
\end{proof}

\section{Solution retrieval}
\label{s:retrieval}
In this section we explain how to retrieve an approximation
$\tilde{x}$ of the optimal solution $x^*$ of problem $P$. Let
$\delta>0 $, by Theorem \ref{thm:objfunapprox}, we can build a vector
$x' \in \mathbb{R}^n_+$ such that $v(P)-\delta \le c\, x'\le v(P)$ and
$T Ax'=T b$ for some $k \times m$ projection matrix $T$.

\subsection{Infeasibility of projected solutions}
We first prove that \LL{$Ax'\not=b$ almost surely}, which means
that the projected problem directly gives us an approximate optimal
objective function value, but not the optimum itself. Let $0 \le \nu
\le \delta$ such that $v(P_T)=v(P)-\nu$.

Let $\tilde{A}=\left(\begin{array}{cc} c\\ A \end{array}\right)$,
$\tilde{b}=\left(\begin{array}{c} v(P)-\nu\\ b\end{array}\right)$, and
$\tilde{T}=\left(\begin{array}{c} 1 \\ T\end{array}\right)$.
We assume here that the projected solution $x'$ (s.t.~$cx'= v(P)-\nu$) is
found uniformly in the projected solution set $F'=\{x \in\mathbb{R}^n_+ \; |\;
\tilde{T}\tilde{A}x=\tilde{T}\tilde{b}\}$. We denote $F=\{x
\in\mathbb{R}^n_+ \; |\; \tilde{A}x=\tilde{b}\}$.

\PL{\begin{proposition}
      \label{prop:certificate1}
  Assume that $\cone{A}$ is full dimensional in $\mathbb{R}^m$ and that any optimal solution of $P$ has at least $m$ non-zero components. Let $x'$ be uniformly chosen in
  $F'$. Then, almost surely, $\tilde{A}x'=\tilde{b}$ does not hold.
\end{proposition}
\begin{proof} 
{\it Proof.} If $\nu>0$ then obviously $\tilde{A}x'=\tilde{b}$ does not hold, because otherwise, it would contradict the minimality of $v(P)$. Hence we assume in the rest of the proof that $\nu=0$, i.e, the value of the projected problem is the same than the value of the original one.

In order to aim at a contradiction, we assume that
\[\mbox{\sf Prob}(x' \in F)=\mbox{\sf p} > 0.\] For each $\epsilon \in \ker{T}$, let
\[F_\epsilon=\{x\ge 0\;|\; \tilde{A}x-\tilde{b}=\epsilon\}\cap F'.\] We will prove
that there exists $d >0$ and a family $\mathcal{V}$ of
infinitely many $\epsilon \in \ker{\tilde{T}}$ such that $\mbox{\sf Prob}(x' \in F_\epsilon)
\ge d >0$. Since $(F_\epsilon)_{\epsilon \in \mathcal{V}}$ is a family of
disjoint sets, we deduce that $\mbox{\sf Prob}\left(x'\in
\bigcup\limits_{\epsilon \in \mathcal{V}}F_v\right) \ge \sum\limits_{\epsilon \in
  \mathcal{V}}d =+\infty$, leading to a contradiction.

\begin{quote}
\textbf{Claim:} $\tilde{b}$ belongs to the relative interior of a facet of the $m+1$ dimensional cone, $\cone{\tilde{A}}$. \\
{\bf Proof of claim}. Notice first that if $\tilde{b}$ belongs to the relative interior of $\cone{\tilde{A}}$ then we can  find a feasible solution for $P$ with a smaller cost. Hence $\tilde{b}$ belongs to a face of dimension at most $m$. Assume now, to aim at a contradiction, that $\tilde{b}$ belongs to the relative interior of a face of dimension $d\le m-1$ of $\cone{\tilde{A}}$. Then, we could write $\tilde{b}$ as a positive sum of $d$ extreme rays, $\tilde{A}_j, \ j \in J$ . Hence there exists an optimal solution $x^*$ of $P$  with $d$ non-negative components. Since $d < m$ there is a  contradiction.
\end{quote}
\noindent Hence $0$ belongs to a facet of $\{\tilde{A}x-\tilde{b}\;|\; x
\ge0\}$, and since $\dim{\ker{\tilde{T}}}\ge 2$ (w.l.o.g.), then there
exists a segment $[-u,u]$ (for $\|u\|$ small enough) that is contained
in the intersection $\ker{\tilde{T}} \cap \{\tilde{A}x-\tilde{b}\;|\;
x \ge0\}$.

Let $\tilde{A}_j,\ j \in J$ be the rays of $\cone{\tilde{A}}$ that
belong to the same facet of $\cone{\tilde{A}}$ as $\tilde{b}$. There
exists $\bar{x}\ge 0$ such that $A\bar{x}=b$ and $\bar{x}_j >
0,\ \forall j \in J$ (because $\tilde{b}$ belongs to the relative
interior of this facet). Since $[-u,u]$ belongs to this facet, there
exits $\hat{x}\in \mathbb{R}^n$ such that $A\hat{x}=-u$ and such that
$\hat{x}_j =0,\ \forall j \notin J$. We can hence compute $\bar{N}>0$
large enough such that $2\hat{x}\le \bar{N} \bar{x}$. 
 
For all $N\ge \bar{N}$ and for all $x\in F$, we denote
$x'_N=\frac{\bar{x}+x}{2}-\frac{1}{N}\hat{x}$. Then we have
$\tilde{A}x'_N=\tilde{b}-\frac{1}{N}\tilde{A}\hat{x}=\tilde{b}+\frac{u}{N}$
and
$x'_N=\frac{x}{2}+(\frac{\bar{x}}{2}-\frac{\hat{x}}{N})\ge0$. Therefore,
\[\frac{\bar{x}+F}{2}-\frac{1}{N}\hat{x} \subseteq F_{\frac{u}{N}}\]
which implies that, for all $N\ge\bar{N}$, 
\[\mbox{\sf Prob}(x' \in F_{\frac{u}{N}})=\mu(F_{\frac{u}{N}})\ge \mu(\frac{\bar{x}+F}{2})\ge \alpha\mu(F)=\alpha\mbox{\sf p}>0\]
for some constant $\alpha>0$, where $\mu$ is a uniform measure on $F'$. \qed
\end{proof}
}

\subsection{Approximate solution retrieval}

\Ky{

Let us consider $y^*$ to be an optimal solution of the following dual problem:
\begin{equation} 
D\equiv\max\; \{b^\top y\;|\;y^\top A\le c\land y\in\mathbb{R}^m\} \label{dual}
\end{equation}		
	and let $y_T$ be an optimal solution of the dual of the projected problem:	
\begin{equation}
D_T\equiv\max\; \{(Tb)^\top \,y\;|\;y^\top\,TA\le c\land
y\in\mathbb{R}^k\}. \label{projdual}
\end{equation}
Let define $y_{\mbox{\sf\scriptsize prox}} = T^\top y_T$. It is easy
to see that $y_{\mbox{\sf\scriptsize prox}}$ is also a feasible
solution for the dual problem $D$ in (\ref{dual}). 

In this section we will assume that
the vector $b \in \mathbb{R}^m$ belongs to the relative interior of the normal cone at some vertex of the dual polyhedron. Under this assumption, the dual solution $y^*$ is uniquely determined.

Let $C_t(y^*)$ be the tangent cone of the dual polyhedron $\mathcal{F}(D) \equiv \{y\in\mathbb{R}^m|\;y^\top A\le c\} $ at $y^*$, which is defined as
\[
C_t(y^*) = \mbox{ \sf closure} \left(\big\{d \; : \exists \lambda > 0 \mbox{ such that } x + \lambda d \in  \mathcal{F}(D) \big\}\right)
\]

In other words, $C_t(y^*)$  is the closure of the set of all feasible directions of the dual polyhedron $\mathcal{F}(D) $ at $y^*$.  Moreover, it is a convex cone generated by a  set of vectors $v^i=y^i-y^*$ where $y^i \mbox{ are the neighboring vertices of } y^*$ for $ i\le p$. Notice that by the previous hypothesis, we have:
\[\transpose{b}v^i<0 \qquad \mbox{ for all } i \le p.
\]
For each $1 \le i  \le p$, let $\alpha_i$ denote the angle between the vectors $-b$ and $v^i$. Let denote by 
\[ \alpha^* \in  \argmin\limits_{\alpha_i, \ldots, \alpha_p} \; \cos(\alpha_i)
\]

We first prove the following lemma, which states that  $y_{\mbox{\sf\scriptsize prox}}$ is approximately close to $y^*$.
	
\begin{lemma}\label{lem:approxdual}
	For any $\varepsilon >0$, there is a constant $\mathcal{C}$ such that:
\begin{equation}\label{eq:2}
\|y^*-y_{\mbox{\sf\scriptsize prox}}\|_2 \le \frac{ \mathcal{C}\theta^2\varepsilon}{\cos(\alpha^*)\|b\|_2} \|y^*\|_2
\end{equation}	
	with probability at least $p= 1 - 4ne^{-\mathcal{C}(\varepsilon^2-\varepsilon^3) k}$
\end{lemma}

\begin{proof}  {\it Proof.}
		By definition, 
		$y_{\mbox{\sf\scriptsize prox}}$ is also a feasible
		solution for the dual problem $D$. Furthermore, by Theorem \ref{thm:objfunapprox},  there is a constant $\mathcal{C}$ such that:
		\begin{equation}\label{eq:1}
		\transpose{b}y_{\mbox{\sf\scriptsize prox}} \ge \transpose{b}y^* -   \mathcal{C}\theta^2 \varepsilon\|y^*\|_2
		\end{equation}
		with probability at least $p= 1 - 4ne^{-\mathcal{C}(\varepsilon^2-\varepsilon^3) k}$.
		
		Since $y_{\mbox{\sf\scriptsize prox}}-y^*$ belongs to the tangent cone $C_t(y^*)$, there exists non-negative scalars $\lambda_i$ (for $i\le p$) such that $y_{\mbox{\sf\scriptsize prox}}-y^*=\sum\limits_{i=1}^{p}\lambda_iv^i$. Hence
	\[ \|y^*-y_{\mbox{\sf\scriptsize prox}}\|_2 = \|\sum\limits_{i=1}^{p}\lambda_iv^i\|_2 \le \sum\limits_{i=1}^{p}\lambda_i\|v^i\|_2 . \]
		By equation \eqref{eq:1}, we have also
		$$ \mathcal{C}\theta^2 \varepsilon\|y^*\|_2  \ge \transpose{b}(y^*-y_{\mbox{\sf\scriptsize prox}}) = \sum\limits_{i=1}^{p}\lambda_i(-\transpose{b}v^i)\quad \mbox{(we recall that $-\transpose{b}v^i >0$ for all $i$) }. $$
		
		Let us consider the following LP:
		
		\begin{equation}
		\left.\begin{array}{llll}
		\max \ & \sum\limits_{i=1}^{p}\lambda_i\|v^i\|_2  && \\
		&\sum\limits_{i=1}^{p}\lambda_i(-\transpose{b}v^i)\le  \mathcal{C}\theta^2 \varepsilon\|y^*\|_2 \\
		&    \lambda \ge  0.
		\end{array}\right\} \label{eq:auxlp}
		\end{equation}
		The LP above is a simple continuous knapsack problem whose solution can be computed easily by a greedy algorithm:
		let $j$ be such that $\frac{\|v^j\|_2}{-\transpose{b}v^j}\ge \frac{\|v^i\|_2}{-\transpose{b}v^i} $ for all $i\in \{1, \ldots, p\}$, then 
		$$\frac{\|v^j\|_2}{-\transpose{b}v^j} \mathcal{C}\theta^2 \varepsilon\|y^*\|_2= \frac{1}{\cos(\alpha^*)\|b\|_2} \mathcal{C}\theta^2 \varepsilon\|y^*\|_2$$ is the optimal value of \eqref{eq:auxlp}. The lemma is proved. \qed
\end{proof}

}

We consider the following algorithm which retrieves an approximate
solution for the original LP from an optimal basis of the projected
problem.
\begin{algorithm}[!ht]   
  \caption{Retrieving an approximate solution of $P$}
  \label{alg2}     
  \begin{algorithmic}
    \State Let $y_T$ be the associated basic dual solution of the projected dual problem $(D_T)$.
    \State Define $y_{\mbox{\sf\scriptsize prox}} = T^\top y_T$
    \For {all $1 \le j \le n$}\\
	 {\qquad \PL{$z_j := \frac{c_j -  A_j^\top y_{\mbox{\sf\scriptsize prox}}}{\|A_j\|_2}  $}
         } \EndFor
	 \State Let $\mathcal{B}$ be the set of indices $j$ corresponding to the $m$ smallest values of $z_j$. \\
	 \Return $x: = A_{\mathcal{B}}^{-1} b$.
  \end{algorithmic}
\end{algorithm}

\PL{Notice that, for all $1\le j \le n$, $z_j := \frac{c_j -  A_j^\top y_{\mbox{\sf\scriptsize prox}}}{\|A_j\|_2}  $ is the distance between $ y_{\mbox{\sf\scriptsize prox}}$ and the hyperplane defined by $A_j^\top y = c_j$. Hence, Algorithm \ref{alg2} searches for the $m$ facets of the dual polyhedron that are the closest to $ y_{\mbox{\sf\scriptsize prox}}$ and return the corresponding basis. \\
	
	Let $\mathcal{B}^*$ be the optimal basis.  We consider the shortest distance from $y^*$ to any hyperplane $A_j^\top y =c_j$ for $j \notin \mathcal{B}^*$:
	$$d^*= \min\limits_{j \notin \mathcal{B}^*}\frac{c_j -  A_j^\top y^*}{\|A_j\|_2}$$	
	       }

\begin{proposition}
\label{propPL}
Assume that the LP problem $P$ satisfies the following two assumptions:
\begin{itemize}
	\item[(a)]  there is no degenerated vertex in the dual polyhedron.
	\item[(b)]  the vector $b \in \mathbb{R}^m$ belongs to the relative interior of the normal cone at some vertex of the dual polyhedron.
\end{itemize}
	\label{prop:solretrieve} 
	\PL{If $$\frac{\mathcal{C}\theta^2 \varepsilon}{\cos(\alpha^*)\|b\|_2} \|y^*\|_2 < \frac{d^*}{2},$$
		where $\mathcal{C}$ is the universal constant in Lemma  \ref{lem:approxdual}, 
then with probability at least $p= 1 - 4ne^{-\mathcal{C}(\varepsilon^2-\varepsilon^3) k}$,  the Algorithm \ref{alg2} returns an optimal basis solution.}
\end{proposition}
\begin{proof}  {\it Proof.}
	\PL{
By Lemma \ref{lem:approxdual}, we have that with probability at least $p= 1 - 4ne^{-\mathcal{C}(\varepsilon^2-\varepsilon^3) k}$,
 $$\|y^*-y_{\mbox{\sf\scriptsize prox}}\|_2 \le \frac{ \mathcal{C}\theta^2\varepsilon}{\cos(\alpha^*)\|b\|_2} \|y^*\|_2$$
 
  Let $\mathcal{B}^*$ be the optimal basis.  Since $\|y^*-y_{\mbox{\sf\scriptsize prox}}\|_2 < \frac{d^*}{2}$, We deduce that for all $j \in \mathcal{B}^*$, $z_j \le \|y^*-y_{\mbox{\sf\scriptsize prox}}\|_2 < \frac{d^*}{2}$.\\

  Now, let us consider $j \notin \mathcal{B}^*$. We have $z_j\ge \frac{d}{2}$, otherwise $y^*$ would be at a distance less than $d^*$ from $A_j^\top y =c_j$. Since $y^*$ is non-degenerated we have $d^*>0$. This ends the proof.
} \qed
\end{proof}

Note that both assumptions (a)-(b) in Prop.~\ref{propPL} hold almost surely for random instances.

\LL{
\section{Computational complexity}
\label{s:complexity}
The main aim of this paper is that of proving that random projections can be applied to the given LP $P$ with some probabilistic bounds on feasibility and optimality errors. The projected LP $P_T$ can be solved by any method, e.g.~simplex or interior point. Formally, we envisage the following the solution methodology:
\begin{quote}
\begin{enumerate}
\item sample a random projection matrix $T$;
\item perform the multiplication $T\,(A,b)$;
\item solve $P_T$ (Eq.~\eqref{proj});
\item retrieve a solution for $P$,
\end{enumerate}
\end{quote}
where $(A,b)$ is the $m\times (n+1)$ matrix consisting of $A$ with the column $b$ appended.

A very coarse computational complexity estimation is as follows: we assume computing each component of $T$ takes $O(1)$, so computing $T$ is $O(km)$. The best practical algorithm for serial matrix multiplication is only very slightly better than the naive algorithm, which takes $O(kmn)=O(mn\log n)$, but more efficient parallel and distributed algorithms exist. For solution retrieval, Alg.~\ref{alg2} runs in time $O(km + mn + n\log n + m^2)=O(n(m+\log n))$. The complexity $O(mn\log n)$ of matrix multiplication therefore dominates the complexity of sampling.

The last step, solution retrieval, is essentially dominated by taking the inverse of the $m\times m$ matrix $A_{\mathcal{B}}$ in Alg.~\ref{alg2}, which we can assume to have complexity $O(m^3)$.

We focus our discussion on the most computationally costly step, i.e.~that of solving the projected LP $P_T$. Exact polynomial-time methods for LP, such as the ellipsoid method or the interior point method, have complexity estimates ranging from $O(n^4 L)$ to $O(\frac{n^3}{\log n} L)$, where $L = \sum_{i=0}^m \sum_{j=0}^n \lceil\log(|a_{ij}|+1)+1\rceil$, $a_{i0} = b_i$ for all $i\le m$, and $a_{0j}=c_j$ for all $j\le n$ \cite{wrightIPM}.

Obviously, these LP complexity bounds are impacted by replacing the number $m$ of rows in $P$ by the corresponding number $k=O(\ln n)$ in $P_T$. Also note that, since $m\le n$, the complexity of solving an LP always exceeds (asymptotically) the complexity of the other steps. So the overall worst-case asymptotic complexity of our solution methodology does not change with respect to solving the original LP. On the other hand, $m$ appears implicitly as part of $L$. If we assume we can write $L$ as $mL'$ for some $L'$, then the complexity goes from $O(\frac{n^3}{\log n} mL')$ to $O(\frac{n^3}{\log n} (\ln n) L')=O(n^3 L')$.

The simplex method has exponential time complexity in the worst case. On the other hand, its average complexity is $O(m n^4)$ \cite[Eq.~(0.5.15)]{borgwardt} in terms of the number of pivot steps, each taking $O(m^2 \tilde{L})$ in a naive implementation \cite{pan}, where $\tilde{L}$ represents a factor due to the encoding length (assumed multiplicative). This yields an overall average complexity bound $O(m^3 n^4 \tilde{L})$. Replacing $m$ by $O(\ln n)$ yields an improvement $O(n^4(\ln n)^3\tilde{L})$.
}

\LL{
\section{Computational results}
\label{s:compres}
A sizable majority of works on the applications of the JLL are theoretical in nature (with some exceptions, e.g.~\cite{venkatasubramanian,quantreg2}). In this section we provide some empirical evidence that our ideas show a rather solid promise of practical applicability.

We started our empirical study by considering the {\sc NetLib} public LP instance library \cite{netlib_lp}, but it turns out that its instances are too small and sparse to yield any CPU improvement. We therefore decided to generate and test a set of random LP instances in standard form. Our test set consists of 360 infeasible LPs and 360 feasible LPs. We considered pairs $(m,n)$ as shown in Table \ref{t:1}.
\begin{table}[!ht]
  \begin{center}
  {\small
    \begin{tabular}{l|rrr}
      & \multicolumn{3}{c}{$m$} \\
      & 500 & 1000 & 1500 \\ \hline
   & 600 & 1200 & 1800 \\
  $n$ & 700 & 1400 & 2100 \\
   & 800 & 1600 & 2400 
    \end{tabular}
  }
  \end{center}    
  \caption{Instance sizes.}
  \label{t:1}
\end{table}
For each $(m,n)$ we test constraint matrix densities in $\mbox{\sf dens}\in\{0.1,0.3,0.5,0.7\}$. For each triplet $(m,n,\mbox{\sf dens})$ we generate 10 instances where each component of the constraint matrix $A$ is sampled from a uniform distribution on $[0,1]$. The objective function vector is always $c=\mathbf{1}$. Infeasible instances are generated using Farkas' lemma: we sample a dual solution vector $y$ such that $yA\ge 0$ and then choose $b$ such that $by<0$. Feasible instances are generated by sampling a primal solution vector $x$ and letting $b=Ax$.

We employ Achlioptas random projectors in order to decrease the density of the projected constraint matrix. One of the foremost difficulties in using random projections in practice is that the theory behind them gives no hint as regards the ``universal constants'', e.g.~$\mathcal{C}$ and the constant implicit in the definition of $k$ as $O(\frac{1}{\varepsilon^2} \ln n)$. In theory, one should be able to work out appropriate values of $\varepsilon$ and of the number $\sigma$ of samplings of the random projector $T$ for the problems at hand. In practice, following the theory will yield such small $\varepsilon$ and large $\sigma$ values that the smallest LPs where our methodology becomes efficient will be expected to have billions of rows, defying all computation on modest hardware such as today's laptops. In fact, we are defending the point of view that random projections are useful in day-to-day work involving large but not necessarily huge LPs and common hardware platforms. For such LPs, a lot of guesswork and trial-and-error is needed. In our computational results we use $k=\frac{1.8}{\varepsilon^2}\ln n$ after an indication found in \cite{venkatasubramanian}, $\varepsilon=0.2$ after testing some values between 0.1 and 0.3, and $\sigma=1$ again after some testing. The choice $\sigma=1$ implies that, occasionally, a few pairwise distances might fall outside their bounds; but enforcing {\it every} pairwise distance to satisfy the JLL requires excessive amounts of samplings of $T$. Besides, concentration of measure ensures that very few pairwise distances will be projected wrong w.o.p.

All results are obtained using a Julia \cite{julia} JuMP \cite{jump} script calling the CPLEX \cite{cplex126} barrier solver (without crossover) on four virtual cores of a dual core Intel i7-7500U CPU at 2.70GHz with 16GB RAM (we remark that Julia is a just-in-time compiled language, so aside from a small lag to initially compile the script, CPU times should be similar to compiled rather than interpreted programs). The CPLEX barrier solver is, in our opinion, the solver of choice when solving very large and possibly dense LPs --- our preliminary tests with the simplex method showed repeated failures due to excessive resource usage (both CPU and RAM), and high standard deviations in evaluating the computational advantage between original and projected problems. Eliminating the crossover phase is a choice we made after some experimentation with these instances. Some preliminary results with very large quantile regression problems show that this choice may need to be re-evaluated when solving problems with different structures.

\subsection{Infeasible instances}
We benchmark infeasible instances on CPU time and accuracy. The latter is expressed in terms of mismatches: i.e., an infeasible original LP that is mapped into a feasible projected LP (recall that the converse can never happen by linearity). The results are shown in Table \ref{t:2}. Each line is obtained as an average over the 10 instances with same $m,n,\mbox{\sf dens}$. 
\begin{table}[!ht]
  \begin{center}
    {\footnotesize
      \begin{tabular}{rrr|rrrr}
        $m$ & $n$ & \textsf{dens} & $k$ & \textsf{orgCPU} & \textsf{prjCPU} & \textsf{acc} \\ \hline      
        500 & 600 & 0.1 & 289 & \textbf{2.40} & 2.66 & 0.0 \\
        500 & 600 & 0.3 & 289 & \textbf{2.15} & 2.80 & 0.0 \\
        500 & 600 & 0.5 & 289 & \textbf{2.48} & 2.95 & 0.0 \\
        500 & 600 & 0.7 & 289 & \textbf{2.91} & 3.12 & 0.0 \\
        500 & 700 & 0.1 & 296 & \textbf{2.46} & 2.99 & 0.0 \\
        500 & 700 & 0.3 & 296 & \textbf{2.24} & 2.93 & 0.0 \\
        500 & 700 & 0.5 & 296 & \textbf{2.72} & 3.34 & 0.0 \\
        500 & 700 & 0.7 & 296 & 3.49 & \textbf{3.38} & 0.0 \\
        500 & 800 & 0.1 & 302 & \textbf{2.01} & 3.11 & 0.0 \\
        500 & 800 & 0.3 & 302 & \textbf{2.35} & 3.17 & 0.0 \\
        500 & 800 & 0.5 & 302 & \textbf{2.95} & 3.58 & 0.0 \\
        500 & 800 & 0.7 & 302 & \textbf{3.60} & 3.95 & 0.0 \\
        1000 & 1200 & 0.1 & 321 & 5.47 & \textbf{4.50} & 0.0 \\
        1000 & 1200 & 0.3 & 321 & 6.92 & \textbf{5.76} & 0.0 \\
        1000 & 1200 & 0.5 & 321 & 9.54 & \textbf{6.87} & 0.0 \\
        1000 & 1200 & 0.7 & 321 & 13.75 & \textbf{7.79} & 0.0 \\
        1000 & 1400 & 0.1 & 327 & \textbf{5.34} & 5.40 & 0.0 \\
        1000 & 1400 & 0.3 & 327 & 7.89 & \textbf{6.48} & 0.0 \\
        1000 & 1400 & 0.5 & 327 & 12.02 & \textbf{8.47} & 0.0 \\
        1000 & 1400 & 0.7 & 327 & 20.93 & \textbf{9.73} & 0.0 \\
        1000 & 1600 & 0.1 & 333 & \textbf{5.64} & 6.29 & 0.0 \\
        1000 & 1600 & 0.3 & 333 & 8.26 & \textbf{8.23} & 0.0 \\
        1000 & 1600 & 0.5 & 333 & 13.20 & \textbf{10.15} & 0.0 \\
        1000 & 1600 & 0.7 & 333 & 20.26 & \textbf{13.34} & 0.0 \\
        1500 & 1800 & 0.1 & 339 & \textbf{7.40} & 8.04 & 0.0 \\
        1500 & 1800 & 0.3 & 339 & 14.38 & \textbf{10.84} & 0.0 \\
        1500 & 1800 & 0.5 & 339 & 24.83 & \textbf{13.97} & 0.0 \\
        1500 & 1800 & 0.7 & 339 & 41.98 & \textbf{19.02} & 0.0 \\
        1500 & 2100 & 0.1 & 346 & \textbf{7.98} & 10.05 & 0.0 \\
        1500 & 2100 & 0.3 & 346 & 17.27 & \textbf{12.20} & 0.0 \\
        1500 & 2100 & 0.5 & 346 & 33.35 & \textbf{16.27} & 0.0 \\
        1500 & 2100 & 0.7 & 346 & 66.81 & \textbf{19.72} & 0.0 \\
        1500 & 2400 & 0.1 & 352 & \textbf{8.52} & 13.54 & 0.0 \\
        1500 & 2400 & 0.3 & 352 & 20.00 & \textbf{17.78} & 0.0 \\
        1500 & 2400 & 0.5 & 352 & 39.01 & \textbf{24.75} & 0.0 \\
        1500 & 2400 & 0.7 & 352 & 65.85 & \textbf{31.95} & 0.0 \\
        \hline
      \end{tabular}
    }
  \end{center}
  \caption{Results on infeasible instances.}
  \label{t:2}
\end{table}
We denote by $m$ the number of rows, by $n$ the number of columns, and by \textsf{dens} the properties of the constraint matrix $A$. We then report the number of rows $k$ in the projected problem, the time \textsf{orgCPU} taken to solve the original LP, the time \textsf{prjCPU} taken to solve the projected LP, and the accuracy \textsf{acc} (``zero'' means that no instance was incorrectly classified as feasible in the projection). While for smaller instances the proposed methodology is not competitive as regards the CPU time, the trend clearly shows that the larger the size of the orginal LP, the higher the chances of our methodology being faster, in accordance with theory. We remark that \textsf{prjCPU} is the sum of the times taken to sample $T$, to perform the matrix multiplication $TA$,  and to solve the projected problem. 

\subsection{Feasible instances}

Feasible instances are benchmarked on CPU time as well as on three discrepancy measures to ascertain the quality of the approximated solution $x^\ast$ of the projected LP. In particular, we look at feasibility with respect to both $Ax=b$ and $x\ge 0$, as well as at the optimality gap between the approximate and the guaranteed optimal objective function value. Unfortunately, we found very high errors in the application of the solution retrieval method in Alg.~\ref{alg2}, which we are only able to justify by claiming our test LPs are ``too small''. We therefore also tested a different solution retrieval method based on the pseudoinverse: it consists in replacing $A_{\mathcal{B}} x = b$ (see last line of Alg.~\ref{alg2}) by the reduced system $\transpose{A}_{\mathcal{H}}A_{\mathcal{H}} x = \transpose{A}_{\mathcal{H}}b$, where $\mathcal{H}$ is a basis of the projected problem $P_T$ (the reconstruction of the full solution from the projected basic components is heuristic). Accordingly, we present two sets of statistics for feasible instances: one labelled ``1'', referring to Alg.~\ref{alg2}, and the other labelled ``2'', referring to the pseudoinverse variant.

The results on the feasible instances are given in Table \ref{t:2}. Again, each line is obtained as an average over the 10 instances with same $m,n,\mbox{\sf dens}$. The CPU time comparison takes three columns: \textsf{orgCPU} refers to the time taken by CPLEX to solve the original LP; \textsf{prjCPU1} is the sum of the times taken to sample $T$, multiply $T$ by $A$, solve the projected LP, and retrieve the original solution by Alg.~\ref{alg2}; and \textsf{prjCPU2} is the same as \textsf{prjCPU1} but using the solution retrieval method based on the pseudoinverse. The solution quality is evaluated in the six columns \textsf{feas1}, \textsf{feas2} (verifying feasibility with respect to $Ax=b$ using the two retrieval methods), \textsf{neg1}, \textsf{neg2} (verifying feasibility with respect to $x\ge 0$ using the two retrieval methods), and \textsf{obj1}, \textsf{obj2} (evaluating the optimality gap using the two retrieval methods), defined as follows:
\begin{itemize}
  \item $\mbox{\sf feas}=\frac{1}{\|b\|_1} \sum\limits_{i\le m}|A^i x^\ast-b_i|$;
  \item $\mbox{\sf neg}=\frac{1}{\|x^\ast\|_1} \sum\limits_{x^\ast_j<0} |x_j^\ast|$;
  \item $\mbox{\sf obj}=\frac{|v(P)-v(P_T)|}{|v(P)|}$.
\end{itemize}

The results are presented in Table \ref{t:3}.
\begin{table}[!ht]
  \begin{center}
    {\footnotesize
      \hspace*{-0.5cm}\begin{tabular}{rrr|r|rrr|rr|rr|rr}
        $m$ & $n$ & \textsf{dens} & $k$ & \textsf{orgCPU} & \textsf{prjCPU1} & \textsf{prjCPU2} & \textsf{feas1} & \textsf{feas2} & \textsf{neg1} & \textsf{neg2} & \textsf{obj1} & \textsf{obj2} \\ \hline
        500 & 600 & 0.1 & 289 & \textbf{2.42} & 9.97 & 6.76 & 0.000 & 0.000 &  0.437 & \textbf{0.033} & 0.079 & \textbf{0.055} \\
        500 & 600 & 0.3 & 289 & \textbf{2.41} & 10.24 & 7.08 & 0.000 & 0.000 &  0.442 & \textbf{0.035} & 0.029 & 0.027 \\
        500 & 600 & 0.5 & 289 & \textbf{3.06} & 10.53 & 7.37 & 0.000 & 0.000 &  0.444 & \textbf{0.037} & 0.023 & 0.020 \\
        500 & 600 & 0.7 & 289 & \textbf{3.89} & 10.81 & 7.72 & 0.000 & 0.000 &  0.454 & \textbf{0.036} & 0.042 & \textbf{0.014} \\
        500 & 700 & 0.1 & 296 & \textbf{2.53} & 10.41 & 7.11 & 0.000 & 0.000 &  0.467 & \textbf{0.039} & 0.246 & \textbf{0.050} \\
        500 & 700 & 0.3 & 296 & \textbf{2.46} & 10.72 & 7.58 & 0.000 & 0.000 &  0.453 & \textbf{0.045} & 0.068 & \textbf{0.025} \\
        500 & 700 & 0.5 & 296 & \textbf{3.43} & 11.10 & 7.97 & 0.000 & 0.000 &  0.475 & \textbf{0.043} & 0.065 & \textbf{0.017} \\
        500 & 700 & 0.7 & 296 & \textbf{4.45} & 11.40 & 8.45 & 0.000 & 0.000 &  0.468 & \textbf{0.038} & 0.028 & \textbf{0.012} \\
        500 & 800 & 0.1 & 302 & \textbf{2.01} & 10.67 & 7.58 & 0.000 & 0.000 &  0.472 & \textbf{0.059} & 0.102 & \textbf{0.045} \\
        500 & 800 & 0.3 & 302 & \textbf{2.55} & 11.10 & 8.02 & 0.000 & 0.000 &  0.463 & \textbf{0.060} & 0.053 & \textbf{0.023} \\
        500 & 800 & 0.5 & 302 & \textbf{3.69} & 11.60 & 8.48 & 0.000 & 0.000 &  0.474 & \textbf{0.061} & 0.068 & \textbf{0.015} \\
        500 & 800 & 0.7 & 302 & \textbf{5.03} & 12.03 & 9.02 & 0.000 & 0.000 &  0.473 & \textbf{0.054} & 0.044 & \textbf{0.011} \\
        1000 & 1200 & 0.1 & 321 & \textbf{6.49} & 14.03 & 10.04 & 0.000 & 0.000 &  0.466 & \textbf{0.012} & \textbf{0.036} & 0.067 \\
        1000 & 1200 & 0.3 & 321 & \textbf{9.16} & 15.82 & 11.61 & 0.000 & 0.000 &  0.468 & \textbf{0.012} & 0.054 & \textbf{0.030} \\
        1000 & 1200 & 0.5 & 321 & 14.71 & 17.52 & \textbf{13.46} & 0.000 & 0.000 &  0.487 & \textbf{0.013} & 0.277 & \textbf{0.021} \\
        1000 & 1200 & 0.7 & 321 & 26.89 & 19.45 & \textbf{14.44} & 0.000 & 0.000 &  0.464 & \textbf{0.013} & 0.092 & \textbf{0.014} \\
        1000 & 1400 & 0.1 & 327 & \textbf{6.88} & 15.54 & 11.50 & 0.000 & 0.000 &  0.484 & \textbf{0.013} & 0.222 & \textbf{0.058} \\
        1000 & 1400 & 0.3 & 327 & \textbf{10.34} & 17.05 & 12.91 & 0.000 & 0.000 &  0.495 & \textbf{0.016} & 0.411 & \textbf{0.026} \\
        1000 & 1400 & 0.5 & 327 & 22.80 & 19.85 & \textbf{16.23} & 0.000 & 0.000 &  0.488 & \textbf{0.013} & 0.144 & \textbf{0.016} \\
        1000 & 1400 & 0.7 & 327 & 34.73 & 21.64 & \textbf{16.47} & 0.000 & 0.000 &  0.484 & \textbf{0.013} & 0.111 & \textbf{0.012} \\
        1000 & 1600 & 0.1 & 333 & \textbf{7.16} & 16.98 & 12.93 & 0.000 & 0.000 &  0.487 & \textbf{0.021} & 0.857 & \textbf{0.056} \\
        1000 & 1600 & 0.3 & 333 & \textbf{11.39} & 20.11 & 15.93 & 0.000 & 0.000 &  0.480 & \textbf{0.016} & 0.102 & \textbf{0.021} \\
        1000 & 1600 & 0.5 & 333 & 25.44 & 22.42 & \textbf{18.73} & 0.000 & 0.000 &  0.486 & \textbf{0.017} & 0.073 & \textbf{0.014} \\
        1000 & 1600 & 0.7 & 333 & 40.84 & 26.31 & \textbf{21.28} & 0.000 & 0.000 &  0.483 & \textbf{0.016} & 0.066 & \textbf{0.010} \\
        1500 & 1800 & 0.1 & 339 & \textbf{9.77} & 21.64 & 15.68 & 0.000 & 0.000 &  0.479 & \textbf{0.005} & 0.069 & 0.064 \\
        1500 & 1800 & 0.3 & 339 & 20.81 & 26.33 & \textbf{18.89} & 0.000 & 0.000 &  0.477 & \textbf{0.004} & 0.042 & \textbf{0.027} \\
        1500 & 1800 & 0.5 & 339 & 42.95 & 29.95 & \textbf{22.36} & 0.000 & 0.000 &  0.473 & \textbf{0.004} & 0.054 & \textbf{0.018} \\
        1500 & 1800 & 0.7 & 339 & 74.23 & 35.63 & \textbf{27.82} & 0.000 & 0.000 &  0.472 & \textbf{0.005} & 0.016 & 0.013 \\
        1500 & 2100 & 0.1 & 346 & \textbf{10.38} & 24.78 & 19.02 & 0.000 & 0.000 &  0.485 & \textbf{0.007} & 0.095 & \textbf{0.057} \\
        1500 & 2100 & 0.3 & 346 & 25.74 & 29.22 & \textbf{21.88} & 0.000 & 0.000 &  0.487 & \textbf{0.007} & 0.156 & \textbf{0.022} \\
        1500 & 2100 & 0.5 & 346 & 52.21 & 34.06 & \textbf{26.06} & 0.000 & 0.000 &  0.483 & \textbf{0.007} & 0.046 & \textbf{0.015} \\
        1500 & 2100 & 0.7 & 346 & 90.18 & 36.81 & \textbf{29.58} & 0.000 & 0.000 &  0.487 & \textbf{0.005} & 0.064 & \textbf{0.010} \\
        1500 & 2400 & 0.1 & 352 & \textbf{11.26} & 27.90 & 22.12 & 0.000 & 0.000 &  0.485 & \textbf{0.006} & 0.121 & \textbf{0.050} \\
        1500 & 2400 & 0.3 & 352 & 29.85 & 35.97 & \textbf{28.58} & 0.000 & 0.000 &  0.485 & \textbf{0.006} & 0.134 & \textbf{0.019} \\
        1500 & 2400 & 0.5 & 352 & 61.25 & 42.47 & \textbf{34.99} & 0.000 & 0.000 &  0.489 & \textbf{0.006} & 0.253 & \textbf{0.011} \\
        1500 & 2400 & 0.7 & 352 & 104.58 & 49.98 & \textbf{43.00} & 0.000 & 0.000 &  0.492 & \textbf{0.006} & 0.126 & \textbf{0.008} \\
        \hline
      \end{tabular}
    }
  \end{center}
  \caption{Results on feasible instances.}
  \label{t:3}
\end{table}
Again, we see an encouraging trend showing that the CPU time for creating and solving the projected LP becomes smaller than the time taken to solve the original LP as size and density increase. According to our theoretical development, increasing size/density further will give a definite advantage to our methodology based on random projections. It is clear that feasibility w.r.t.~$Ax=b$ is never a problem. On the other hand, feasibility w.r.t.~non-negativity is an issue, expected with the pseudoinverse-based solution retrieval method, but not necessarily with Alg.~\ref{alg2}. After checking it (and its implementation) multiple times, we came to two possible conclusions: (i) that our arbitrary choice of universal constants is wrong for Alg.~\ref{alg2}, which would require larger instances than those we tested in order to work effectively; (ii) that the choice of the basis $\mathcal{B}$ in Alg.~\ref{alg2} is heavily affected by numerical errors, and therefore wrong. We have been unable to establish which of these reasons is most impactful, and delegate this investigation to future work. For the time being, we propose the pseudoinverse variant as the method of choice.
}

\LL{
\section{An application to error correcting codes}
\label{s:coding}
In this section we showcase an application of our methodology to a problem of error correcting coding and decoding \cite[\S 8.5]{matousek-lp}.

  A binary word $w$ of length $m$ can be encoded as a word $z$ of length $n$ (with $m<n$) such that $z=Qw$ where $Q$ is an $n\times m$ real matrix, which we assume to have rank $m$. After transmission on an analogue noisy channel the other party receives $\bar{z}$. We assume $\bar{z}=z+\bar{x}$, where the transmission error $\bar{x}_j$ on the $j$-th character is uniformly distributed in $[-\delta,\delta]$ for some given $\delta>0$ with some given (reasonably small) probability $\epsilon>0$, and $\bar{x}_j=0$ with probability $1-\epsilon$. In other words, $x$ is a sparse vector with density $\epsilon$.

  The decoding of $\bar{z}$ into $w$ is carried out as follows. We find an $m\times n$ matrix $A$ orthogonal to $Q$ (so $AQ=0$), we compute $b=A\bar{z}$ and note that 
\[ b=A\bar{z}=A(z+x)=A(Qw+x) = AQw + Ax = Ax.\]
If the system $Ax=b$ can be solved, we can find $z'=\bar{z}-x$, and recover $w$ using the projection matrix $(\transpose{Q}Q)^{-1}\transpose{Q}$ followed by rounding:
\[ w=\lfloor(\transpose{Q}Q)^{-1}\transpose{Q}z'\rceil. \]

The protocol rests on finding a sparse solution of the under-determined linear system $Ax=b$. Minimizing the number of non-zero components of a vector that also satisfies $Ax=b$ is known as ``zero-norm minimization'', and is {\bf NP}-hard \cite{natarajan}. In a celebrated discovery later called {\it compressed sensing}, Cand\`es, Rohmberg, Tao and Donoho discovered that the zero-norm is well approximated by the $\ell_1$-norm. We therefore consider the following problem
  \[\min \{ \| x \|_1 \;|\; Ax=b\},\]
which can be readily reformulated to the LP
\begin{equation}
  \min \{ \sum_j s_j \;|\; -s\le x\le s \land Ax=b\}. \label{ecc}
\end{equation}
We propose to compare the solution of Eq.~\eqref{ecc} with that of its randomly projected version:
\begin{equation}
  \min \{ \sum_j s_j \;|\; -s\le x\le s \land TAx=Tb\}, \label{eccT}
\end{equation}
where $T$ is an Achlioptas random projector. The computational set-up for this test is similar to that of Sect.~\ref{s:compres}, except that we enable the crossover in the CPLEX barrier solver.

We compare Eq.~\eqref{ecc} and Eq.~\eqref{eccT} on the sentence that the Sybilla of Delphos spoke to the hapless soldier who asked her whether he would get back from the war or die in it: {\it Ibis redibis non morieris in bello} [Alberico delle Tre Fontane, {\it Chronicon}], at which the soldier rejoiced. When his wife heard he died in the war, she contacted the Sybilla for a full refund. The Sybilla, unperturbed, pointed out that the small print in the legal terms attributed her the right of inserting commas in sentences as she saw fit, which made her prophecy into the more reality-oriented {\it Ibis redibis non, morieris in bello}. We test here the comma-free version, much more cryptic, ambiguous, and therefore worthy of the Sybilla.

The original sentence is encoded in ASCII-128 and then in binary without padding (1001001 1100010 1101001 1110011 1000001 1100101 1001011 1001001 1010011 1000101 1010011 1100111 0000011 0111011 0111111 0111010 0000110 1101110 1111111 0010110 1001110 0101111 0010110 1001111 0011100 0001101 0011101 1101000 0011000 1011001 0111011 0011011 0011011 11). The binary string has $m=233$ characters, is encoded into $n=256$ characters (assuming an error rate of 10\%, typical of the Sybilla muttering incantations with low and guttural voice), and is then projected into $k=61$ characters. We modified the parameter of the Achlioptas projector from 1/6 down to 1/100 after verifying with many examples that this particular application is extremely robust to random projections. 

While the original LP took 0.296s to solve, the projected LP only took 0.028s. The accuracy in retrieving the original text was perfect. In fact, in this application it is very hard to make mistakes in the recovery; so much so, that we could set the JLL $\varepsilon$ at 0.3. This might be partly due to the fact that the LP in Eq.~\eqref{ecc} does not include nonnegativity constraints, which are generally problematic because of their large Gaussian width, see Sect.~\ref{s:litrev}.

We also tested a slightly longer word sequence from a well-known poem about aviary permanence on greek sculptures: {\it Once upon a midnight, dreary, while I pondered, weak and weary}. The 421 characters long binary string is encoded into 463 characters and projected into 67. The original LP took 1.332s and the projected LP took 0.064s to solve; again, the retrieval accuracy was perfect.
}

\section{Conclusion}
\label{s:conclusion}
This paper is about the application of random projections to LP in standard form. We prove that feasibility and optimality are both approximately preserved by sub-gaussian random projections. Moreover, we show how to retrieve solutions of the original LPs from those of the projected LPs, using duality arguments. These findings make it possible to approximately solve very large scale LPs with high probability, \LL{as showcased by our computational results and application to error correcting codes.}

\section*{Acknowledgments}
We are grateful to A.A.~Ahmadi and D.~Malioutov for introducing us to the Johnson-Lindenstrauss lemma, and to F.~Tardella for helpful discussions. We are also grateful to two anonymous referees for helping us to improve the paper. This research is partly supported by the SO-grid project (\url{www.so-grid.com}) funded by ADEME, by the EU Grant FP7-PEOPLE-2012-ITN No.~316647 ``Mixed-Integer Nonlinear Optimization'', and by the ANR ``Bip:Bip'' project under contract ANR-10-BINF-0003. The first author is supported by a Microsoft Research Ph.D.~fellowship.

\ifmor
\else
\bibliographystyle{plain}
\fi
\bibliography{dr2}

\end{document}